\documentclass[final]{cmslatex}

\usepackage{url}
\usepackage{hyperref}
\usepackage{bm}
\usepackage{amsmath}
\usepackage{amssymb}
\usepackage{graphicx}
\usepackage{color}
\usepackage{amsfonts}
\usepackage{amscd}
\usepackage{mathrsfs}
\usepackage{enumerate}
\usepackage{algorithmic, algorithm}
\usepackage{empheq}
\usepackage{tikz}

\newcommand{\p}{\partial}

\def\d{\ensuremath{\mathrm{d}}}
\newcommand{\dps}{\displaystyle}

\newcommand{\pd}[2]{\dfrac{\partial #1}{\partial #2}}

\newcommand{\norm}[1]{\left\|#1\right\|}

\newcommand{\xx}{{\mathbf{x}}}
\newcommand{\yy}{{\mathbf{y}}}
\newcommand{\zz}{{\mathbf{z}}}

\newcommand{\nn}{{\mathbf{n}}}

\newcommand{\RR}{{\mathbb{R}}}

\newcommand{\Rd}{{R_{\delta}(\xx,\yy)}}
\newcommand{\bRd}{{\bar{R}_{\delta}(\xx,\yy)}}
\newcommand{\bbRd}{{\bar{\bar{R}}_{\delta}(\xx,\yy)}}

\newcommand{\wdd}{w_\delta}
\newcommand{\bwd}{\bar{w}_\delta}
\newcommand{\ONL}{{\Omega_{NL}}}
\newcommand{\OL}{{\Omega_{L}}}
\newcommand{\uNL}{u_{NL}}
\newcommand{\uL}{u_{L}}
\newcommand{\uG}{u_{\Gamma}}
\newcommand{\vNL}{v_{NL}}
\newcommand{\vL}{v_{L}}
\newcommand{\eG}{e_{\Gamma}}
\newcommand{\eNL}{e_{NL}}
\newcommand{\eL}{e_{L}}
\newcommand{\bbuNL}{\bar{\bar{u}}_{NL}}
\newcommand{\buNL}{\bar{u}_{NL}}

\newcommand{\bbeNL}{\bar{\bar{e}}_{NL}}
\newcommand{\vG}{v_{\Gamma}}

\newenvironment{equationa*}{\begin{equation*}\begin{aligned}} {\end{aligned}\end{equation*}}
\allowdisplaybreaks[2]

\begin{document}

\title{A seamless local-nonlocal coupling diffusion model with $H^1$ vanishing nonlocality convergence \thanks{This work was supported by National Natural Science Foundation of China (NSFC) 12071244, 92370125.}}


\author{
Yanzun Meng
\thanks{Department of Mathematical Sciences, Tsinghua University 
Beijing, China, 100084. \textit{Email: myz21@mails.tsinghua.edu.cn}
}
\and
Zuoqiang Shi
\thanks{Corresponding Author, Yau Mathematical Sciences Center, Tsinghua University 
Beijing, China, 100084. \&
Yanqi Lake Beijing Institute of Mathematical Sciences and Applications
 Beijing, China, 101408. \textit{Email: zqshi@tsinghua.edu.cn}}
}

\maketitle

\begin{abstract}
	Based on the development in dealing with nonlocal boundary conditions, 
	we propose a seamless local-nonlocal coupling diffusion model in this paper.  
	In our model, a finite constant interaction horizon is equipped in the nonlocal part and transmission conditions are imposed on a co-dimension one interface. 
	To achieve a seamless coupling, we introduce an auxiliary function to merge the nonlocal model with the local part and design a proper coupling transmission condition to ensure the stability and convergence. In addition, by introducing bilinear form, well-posedness of the proposed model can be proved and convergence to a standard elliptic transmission model with first order in $H^1$ norm can be derived. 
\end{abstract}



\section{Introduction}
Nonlocal models are widely used in a number of scientific and engineering fields. Compared with conventional local models, which use differential operators 
to describe some mechanisms under rigorous regularity assumptions, nonlocal models introduce integral operators to characterize more singular phenomenons. For instance, 
in peridynamics \cite{askari2008peridynamics,oterkus2012peridynamic,silling2010crack}, nonlocal models works effectively when there are fracture, mixture or defect in the materials. Additionally, in the context of diffusion \cite{bucur2016nonlocal,vazquez2012nonlinear}, nonlocal models can also describe some anomalous condition. 
Beyond modeling physical system, nonlocal models also attract attentions in some emerging field, such as semi-supervised learning \cite{shi2017weighted,tao2018nonlocal,wang2018non} and imaging process \cite{osher2017low}.

While nonlocal models show their advantages in characterizing complicated mechanisms and improving accuracy in some tasks, the computational cost of solving nonlocal problem is much higher than solving its local counterpart. Nevertheless, the singularity part, which have to be handled with nonlocal models, can often be confined in a small patch that can be identified from the regular part. Therefore, it is natural to use nonlocal models only in 
the singular subdomain and leave the remaining part being described by local model, which is usually partial differential equations. Based on this idea, we can expect a local-nonlocal coupling model to combine accuracy and computational efficiency.

However, it is definitely not simple to couple the distinctly different local and nonlocal descriptions. In fact, compared with classical partial differential equations, it is inconvenient to impose boundary conditions in nonlocal models. In general, a parameter in nonlocal models, which is usually called interaction horizon, should be properly selected according to the specific problem. 
Some information in the interaction horizon may be missing near the boundary. Unintended error will be introduced without proper boundary condition \cite{du2015integral}. 
Therefore, some elaborate designs should be proposed when merging the local and nonlocal models.

A number of attempts have achieved success in the past decades.  In essence, coupling methods are based on the approaches to impose boundary conditions in nonlocal models. 
One popular way prescribes a nonlocal analogue of boundary condition, which is usually called volumetric constraints \cite{du2012analysis}. More information on a collar surrounding the domain is equipped in nonlocal model. 
Thus, in the context of local-nonlocal coupling, the transmission conditions are imposed in a transition region rather than on the interface. 
In \cite{d2021formulation,d2016coupling}, an optimization-based coupling strategy is proposed. 
This method preset the boundary data in the transition region as the variables to be optimized. 
With these information, both nonlocal and local problem can be solved independently with the existing methods.
The coupling system is ultimately achieved by minimize the error in the overlap region of local and nonlocal part via selecting the optimal preset data. 
The well-posedness and convergence analysis of this optimization-based method are provided in \cite{d2021formulation}, but the convergence depends on the thickness of the transition region. 
Following the success in coupling two nonlocal models with different interaction horizons \cite{li2017quasi}, a novel quasi-nonlocal coupling method are proposed in \cite{du2018quasi} to merge a nonlocal diffusion model with its local counterpart. 
The prior work \cite{li2017quasi} adopts the idea of geometric reconstruction \cite{weinan2006uniform} in the subregion dominated by a smaller horizon. Since the local model can be seen as its nonlocal counterpart with zero interaction horizon, \cite{du2018quasi} extends the geometric reconstruction in local part to make the models consistent. 
Nevertheless, the rigorous analysis is limited to one dimension in \cite{du2018quasi}. 
To avoid the transmission via transition region, \cite{you2020asymptoticallycoupling} gives a partitioned coupling framework in dynamic diffusion problem and shows the convergence numerically. 
Although the boundary data is transmitted only at the interface, solving the nonlocal part still need volumetric constraints in \cite{you2020asymptoticallycoupling}. In order to achieve the intrinsic seamless coupling, spatially varying horizon is introduced. If the nonlocal model is equipped with a shrinking horizon as the points approach the boundary, a trace theorem for the corresponding nonlocal space can be established in \cite{tian2017trace}. 
Once the trace is well-defined, the subsequent coupling method \cite{tao2019nonlocal} naturally follows.  In addition, besides the success in nonlocal diffusion as mentioned above, similar coupling strategies can also be applied in nonlocal mechanics \cite{bobaru2016handbook,yu2018partitioned,han2012coupling,wang2019concurrent,lubineau2012morphing,han2016morphing,seleson2013force,seleson2015concurrent}.

In this paper, we illustrate our local-nonlocal coupling diffusion model via approximating elliptic transmission problem. The transmission conditions of this problem include Neumann and Dirichlet constraints on the interface.
In fact, besides volumetric constraints and shrinking horizon, boundary conditions can be imposed via modifying the original nonlocal operators.  
For nonlocal diffusion model, point integral method \cite{li2017point,shi2017convergence} uses the following equation 
\begin{align*}
  \dfrac{2}{\delta^2}\int_{\Omega} \Rd (u(\xx)-u(\yy))\d\xx-\int_{\partial\Omega}\bRd \pd{u}{\nn}(\yy)\d S_\yy=\int_{\Omega}\bRd f(\yy)\d\yy
\end{align*}
to approximate the Poisson equation with Neumann boundary condition. Where $R_\delta$, $\bar{R}_\delta$ are integral kernels defined in Section \ref{sec:Local-nonlocal coupling model and results} and $f$ is the source term in Poisson equation. 
Based on this method, we introduce an auxiliary function to cover the Neumann data and design an additional constraint on the interface to force the system satisfying the Dirichlet continuity constraint.

The rest of this paper is organized as following. In Section \ref{sec:Elliptic transmission problem}, we recall the configuration of elliptic transmission problem. 
Some notations about our nonlocal model are introduced in Section \ref{sec:Local-nonlocal coupling model and results}. More importantly, we derive our local-nonlocal model and give the main results in this section.
The well-posedness of our model is proved in Section \ref{sec:well-posedness} and the convergence analysis is provided in Section \ref{sec:convergence proof}.
\section{Elliptic transmission problem}
\label{sec:Elliptic transmission problem}
In this paper, we assume $\Omega \subset \RR^n$ is an open, bounded domain, and $\p \Omega$ is smooth enough. 
As shown in Fig \ref{Region}, there is a closed $(n-1)$-dimensional smooth manifold $\Gamma\subset \Omega$ splitting $\Omega$ into two regions. 
The region enclosed by $\Gamma$ is denoted as $\ONL$ and another one is denoted as $\OL$.  
\begin{figure}[h]
\centering
\tikzset{every picture/.style={line width=0.75pt}} 

\begin{tikzpicture}[x=0.75pt,y=0.75pt,yscale=-1,xscale=1]

\draw   (162,79.4) .. controls (162,53.77) and (182.77,33) .. (208.4,33) -- (464.6,33) .. controls (490.23,33) and (511,53.77) .. (511,79.4) -- (511,218.6) .. controls (511,244.23) and (490.23,265) .. (464.6,265) -- (208.4,265) .. controls (182.77,265) and (162,244.23) .. (162,218.6) -- cycle ;
\draw   (259,96) .. controls (291,74) and (377,65) .. (407,91) .. controls (437,117) and (456,165) .. (429,197) .. controls (402,229) and (300,229) .. (257,202) .. controls (214,175) and (227,118) .. (259,96) -- cycle ;

\draw (195,63.4) node [anchor=north west][inner sep=0.75pt]    {$\Omega _{L}$};
\draw (422,200.4) node [anchor=north west][inner sep=0.75pt]    {$\Gamma $};
\draw (325,140.4) node [anchor=north west][inner sep=0.75pt]    {$\Omega _{NL}$};
\draw (515,138.4) node [anchor=north west][inner sep=0.75pt]    {$\partial \Omega $};	
\end{tikzpicture}
\caption{A two-dimensional example for our problem. Nonlocal and local model will be applied in $\ONL$ and $\OL$ respectively in our coupling model. The transmission interface is denoted as $\Gamma$. 
Homogeneous Neumann boundary condition is imposed on $\p \Omega$ to simplify the system.}
\label{Region}
\end{figure}
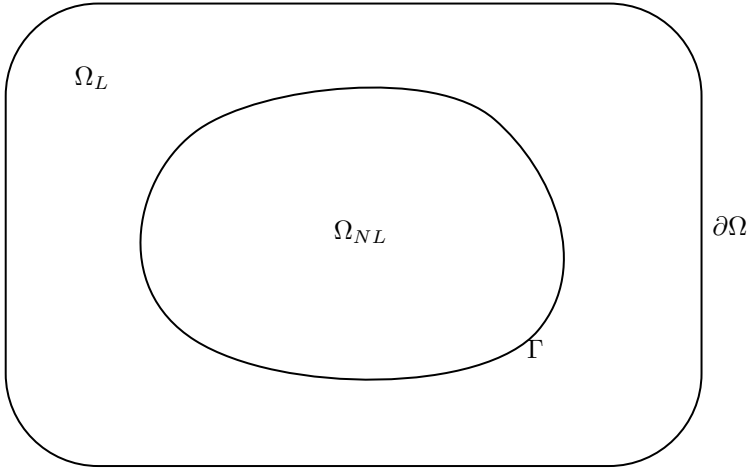 
The configuration of an elliptic transmission problem is in $\Omega=\OL\cup\Gamma\cup \ONL$. In detail, two Poisson equations with different coefficients are imposed in $\OL$ and $\ONL$ respectively. On the interface $\Gamma$, transmission conditions are given to determine the system. Mathematically, the system can be depicted as 
\begin{equation}
\label{original system}
\left\{
\begin{aligned}
  &-\lambda_1\Delta u(\xx)=f(\xx),\quad\xx\in \OL;\\
  &-\lambda_2\Delta u(\xx)=f(\xx),\quad\xx\in \ONL;\\
  &u^+(\xx)=u^{-}(\xx),\ \lambda_1{\pd{u}{\nn}}^+(\xx)=\lambda_2{\pd{u}{\nn}}^-(\xx),\quad \xx\in\Gamma;\\
  &\pd{u}{\nu}(\xx)=0;\quad\xx\in \p \Omega;\\
  &\int_{\Omega}u(\xx)\d\xx=0.
\end{aligned}
\right.
\end{equation}
In (\ref{original system}), $\lambda_1,\lambda_2$ are positive. $\nn(\xx)$ is the unit normal at $\xx\in \Gamma$. 
The direction of $\nn$ is specified from $\ONL$ to $\OL$ consistently. 
In other words, the vector field $\nn$ is the outward unit normal field of $\ONL$ but part of the inner normal field of $\OL$. 
In order to be compatible with the direction of $\nn$, the superscript of $u^+$ and $u^-$  means the limit to $\Gamma$ taken from $\OL$ and $\ONL$ respectively. The same meaning is also applicable to ${\frac{\p u}{\p \nn}}^+$ and ${\frac{\p u}{\p \nn}}^-$.
Additionally, to make the system solvable, a compatibility condition 
\begin{equation}
  \int_\Omega f(\xx)\d\xx=0
\end{equation}
is implied by the system.

This elliptic transmission problem has a weak form. That is to find $u\in H^1(\Omega)$ such that  
\begin{equation*}
  \lambda_1\int_{\OL}\nabla u(\xx)\cdot\nabla v(\xx)\d\xx+\lambda_2\int_{\ONL}\nabla u(\xx)\cdot\nabla v(\xx)\d\xx=\int_{\Omega}f(\xx)v(\xx)\d\xx, \quad\forall v\in H^1(\Omega).
\end{equation*}
Classical partial differential equation theory tells us $f\in H^m(\Omega)$ means there exists unique $u\in H^1(\Omega)\cap H^{m+2}(\OL)\cap H^{m+2}(\ONL)$ satisfying above weak form for $m\geq 0$. 

\section{Local-nonlocal coupling model and results} 
\label{sec:Local-nonlocal coupling model and results}
This section will start from some basic configuration in our nonlocal model. Then a local-nonlocal coupling model will be derived based on point integral method. The main results about this model will also be stated in this section. 
\subsection{Nonlocal kernels}
In this paper, the following assumptions are imposed to a function $R(r)$, which is used to define our nonlocal model.
\begin{itemize}
  \item (smoothness and nonnegativity) $R\in C^1([0,+\infty))$ and $R(r)\geq 0$;
  \item (compact support) $R(r)=0$, for $r\geq 1$;
  \item (nondegeneracy) $\exists \gamma_0>0$ such that $R(r)\geq \gamma_0$ for $0\leq r\leq \frac{1}{2}$.
\end{itemize}
Two functions derived from $R$ are defined as 
\begin{equation*}
  \bar{R}(r)=\int_{r}^{+\infty}R(s)\d s\ \text{ and }\  \bar{\bar{R}}(r)=\int_{r}^{+\infty}\bar{R}(s)\d s.
\end{equation*}
It is simple to verify these two functions also satisfy above assumptions. 
 
Let constant $\alpha_n$ give the normalization
\begin{equation}
  \label{kernel normalization}
  \int_{\RR^n}\alpha_n\delta^{-n}\bar{R}\left(\dfrac{\left|\xx-\yy\right|^2}{4\delta^2}\right)\d\yy=\alpha_n S_n\int_0^2\bar{R}(r^2/4)r^{n-1}\d r=1,
\end{equation}
where $S_n$ is the area of unit ball in $\RR^n$. With this constant, the integral kernels in our nonlocal model are defined as
\begin{equation*}
  \tilde{R}_{\delta}(\xx,\yy)=\alpha_n\delta^{-n}\tilde{R}\left(\dfrac{\left|\xx-\yy\right|^2}{4\delta^2}\right),
\end{equation*}
where $\tilde{R}$ refers to $R$, $\bar{R}$ or $\bar{\bar{R}}$. We list some basic estimations about these kernels, which are heavily used in the following analysis.
\begin{proposition}
  If $\tilde{R}$ refers to $R$, $\bar{R}$ or $\bar{\bar{R}}$, and a $C^2$ domain $U\subset\RR^n$ is open and bounded. We have the following estimations
  \begin{align*}
    C_1\leq \int_U\tilde{R}(\xx,\yy)\d\yy&\leq C_2,\quad \forall\xx\in \bar{U};\\
    \int_{\p U}\tilde{R}(\xx,\yy)\d S_\yy&\leq \frac{C_3}{\delta},\quad \forall\xx\in \bar{U};\\
    \int_{\p U}\tilde{R}(\xx,\yy)\d S_\yy&\geq \frac{C_4}{\delta},\quad \text{when } d(\xx,\p U)<\frac{\sqrt{2}}{2}\delta. 
  \end{align*}
  with $C_i, i=1,2,3,4$ are constants independed of $\delta$.
\end{proposition}
\subsection{Local-nonlocal coupling model}
Point integral method \cite{li2017point,shi2017convergence} provides an integral equation about $u$, ${\frac{\partial u}{\partial \nn}}^+$ and $f$ to approximate Poisson equation. In $\ONL$, if $-\lambda_2\Delta u=f$, we can get  
\begin{align}
  \label{PIM}
  \dfrac{\lambda_2}{\delta^2}\int_{\ONL}\Rd (u(\xx)-u(\yy))\d\xx-2\lambda_2\int_{\Gamma}\bRd {\pd{u}{\nn}}^+(\yy)\d S_\yy\notag\\
  =\int_{\Omega_{NL}}\bRd f(\yy)\d\yy+r_{NL,1}(\xx), \ \xx\in\ONL,
\end{align}
where $r_{NL,1}(\xx)$ is the truncation error. In (\ref{PIM}), the normal derivative on the interface $\Gamma$ should be given, but it is missing in our problem. 
To deal with this issue, we introduce an auxiliary function $u_\Gamma$ defined on $\Gamma$. By replacing  ${\frac{\p u}{\p \nn}}^+$ with $u_\Gamma$ and ignoring the truncation error in (\ref{PIM}),
we can get an integral equation about a function $u_{NL}$ defined in $\ONL$, that is 
\begin{align}
  \dfrac{\lambda_2}{\delta^2}\int_{\ONL}\Rd (\uNL(\xx)-\uNL(\yy))\d\xx-2\lambda_2\int_{\Gamma}\bRd \uG(\yy)\d S_\yy\nonumber\\
  =\int_{\Omega}\bRd f(\yy)\d\yy, \quad \xx\in\ONL.\label{NLequ1}
\end{align}
Naturally, according to the transmission condition $\lambda_1{\frac{\p u}{\p \nn}}^+=\lambda_2{\frac{\p u}{\p \nn}}^-$, 
a Poisson equation with Neumann boundary condition 
\begin{equation}
\label{Leq}
\left\{\begin{aligned}
  -\lambda_1\Delta \uL(\xx)&=f(\xx),\quad\xx\in\OL;\\
  \lambda_1 \pd{\uL}{\nn}(\xx)&=\lambda_2 \uG(\xx),\quad\xx\in\Gamma,
\end{aligned}\right.
\end{equation}
should be imposed in $\OL$. Here $\frac{\p \uL}{\p \nn}$ is in fact the inward normal derivative for $\OL$.

For the additional introduced auxiliary function $\uG$, one extra equation on the interface is necessary. 
The idea is to reapply point integral method for $\xx\in \Gamma$ with different kernels, i.e. 
\begin{align*}
  \lambda_2\int_{\ONL}\bRd (u(\xx)-u(\yy))\d\xx-2\lambda_2\delta^2\int_{\Gamma}\bbRd {\pd{u}{\nn}}^+(\yy)\d S_\yy\notag\\
  =\delta^2\int_{\Omega_{NL}}\bbRd f(\yy)\d\yy+\delta^2\tilde{r}_{NL,1}(\xx), \ \xx\in\Gamma,
\end{align*}
It is reasonable to cut the right-hand side to be $0$ in above equation as $\delta$ is small. Additionally, in the left-hand side,
to derive the coercivity in later analysis, we move the normal derivative out of the integral in the second term to get the equation about $\uNL$ and $\uG$,
\begin{align}
  \lambda_2\int_{\ONL}\bRd (u_{NL}(\xx)-u_{NL}(\yy))\d\yy-2\lambda_2\delta^2\uG(\xx)\int_{\Gamma}\bbRd\d S_\yy=0,\quad \xx\in\Gamma,\label{interfaceequ1}
\end{align}
 Furthermore, to couple the local part and the nonlocal part, using the interface condition $u^+(\xx)=u^-(\xx),\;\xx\in \Gamma$, we get the equation
\begin{align}
  \lambda_2\int_{\ONL}\bRd (\uL(\xx)-\uNL(\yy))\d\yy-2\lambda_2\delta^2\uG(\xx)\int_{\Gamma}\bbRd\d S_\yy=0,\quad \xx\in \Gamma.\label{interfaceequ2}
\end{align}
It is notable that we will find the solution $\uL$ in $H^1(\OL)$ in the following sections, hence the interface value $\uL(\xx)$ makes sense, at least in the sense of trace.

In summary, our local-nonlocal coupling model is an integral equation system about three functions $u_L$, $u_{NL}$ and $u_\Gamma$. If we denote 
\begin{equation}
\label{notation1}
\begin{aligned}
  &{w}_\delta(\xx)=\int_{\ONL}\Rd\d\yy,\qquad \bar{u}_{NL}(\xx)=\dfrac{1}{{w}_\delta(\xx)}\int_{\ONL}\Rd\uNL(\yy)\d\yy,\\
  &\bar{w}_\delta(\xx)=\int_{\ONL}\bRd\d\yy,\qquad \bbuNL(\xx)=\dfrac{1}{\bar{w}_\delta(\xx)}\int_{\ONL}\bRd\uNL(\yy)\d\yy,
\end{aligned}
\end{equation}
and 
\begin{equation}
	\label{right}
  \begin{aligned}
	&f_L(\xx)=f(\xx),\quad\quad f_{NL}(\xx)=\int_{\Omega}\bRd f(\yy)\d\yy+\bar{f},\\
  &\zeta_\delta(\xx)=\dfrac{2\delta^2}{\bar{w}_\delta(\xx)}\int_{\Gamma}\bbRd\d S_\yy,
  \end{aligned}
\end{equation}
our model can be written as 
\begin{equation}
\label{coupling system}
\left\{\begin{aligned}
&-\lambda_1 \Delta u_L(\xx)=f_L(\xx),\quad\xx\in \OL;\\
&\dfrac{\lambda_2}{\delta^2}\int_{\ONL}\Rd (\uNL(\xx)-\uNL(\yy))\d\yy-\lambda_2\int_{\Gamma}\bRd \dfrac{u_\Gamma(\yy)}{\bar{w}_\delta(\yy)}\d S_\yy
=f_{NL}(\xx),\quad\xx\in\ONL;\\
&\lambda_1\pd{\uL}{\nn}(\xx)=\lambda_2\uG(\xx),\quad\xx\in\Gamma;\\
&-\lambda_2\uL(\xx)+\dfrac{\lambda_2}{\bar{w}_\delta(\xx)}\int_{\ONL}\bRd\uNL(\yy)\d\yy+\lambda_2\zeta_\delta(\xx)u_\Gamma(\xx)=0,\quad \xx\in\Gamma.
\end{aligned}\right.
\end{equation}
Notice that we modified some terms compared with (\ref{NLequ1}) and (\ref{interfaceequ2}). 
Firstly, we divided $-\bar{w}_\delta(\xx)$ in (\ref{interfaceequ2}) and replaced the constant coefficient $2\lambda_2$ by a weight function $\frac{\lambda_2}{\bwd(\yy)}$ in (\ref{NLequ1}). 
These two modifications will help us eliminate the cross terms in the bilinear form which will be presented in Section \ref{sec:well-posedness}.
Moreover, the system (\ref{coupling system}) also implies a compatibility condition 
\begin{equation}
\label{compatibility condition}
\int_{\OL} f_{L}(\xx)\d\xx+\int_{\ONL} f_{NL}(\xx)\d\xx=0.
\end{equation}
This equation is achieved by adding a constant $\bar{f}$ in the right-hand side of (\ref{NLequ1}). The estimation of $\left|\bar{f}\right|$ is crucial in the following analysis of well-posedness and convergence. We state the results about this constant here.
\begin{lemma}
\label{LEMMA:F}
Let the auxiliary constant $\bar{f}$ satisfy 
\begin{align*}
  \int_{\OL}f(\xx)\d\xx+\int_{\ONL}\left(\int_{\ONL}\bRd f(\yy)\d\yy+\bar{f}\right)\d\xx=0.
\end{align*}
(1) If $f\in L^2(\Omega)$, we have $\left|\bar{f}\right|\leq C\norm{f}_{L^2(\Omega)}$.\\
(2) If $f\in H^1(\Omega)$, we have $\left|\bar{f}\right|\leq C\delta\norm{f}_{H^1(\Omega)}$.\\
Here $C$ is a constant independent of $\delta$.
\end{lemma}
The proof of above results can be found in Appendix \ref{appendix:proof1}.
\newpage
\subsection{Main results}
We firstly define a space
\begin{align*}
  \hat{H}=\biggl\{(\uL,\uNL,\uG):\uL\in H^1(\Omega_L),\uNL\in L^2(\Omega_{NL}),\uG\in L^2(\Gamma),\\
  \int_{\OL}u_L(\xx)\d\xx+\int_{\ONL}\uNL(\xx)\d\xx=0\biggr\}.
\end{align*}
Now we can give the well-posedness and convergence results about our local-nonlocal coupling model.
\begin{theorem}
  \label{THEOREM:WELL-POSEDNESS}
  Let $f\in L^2(\Omega)$, then our local-nonlocal coupling system (\ref{coupling system}) has a unique solution $(\uL,\uNL,\uG)\in \hat{H}$. Here we say $\uL$ solves (\ref{coupling system}) in 
  the sense of weak solution. Moreover, we have $\uNL\in H^1(\ONL)$ with the following estimation
  \begin{equation}
    \label{well-posedness estimation}
    \norm{\uL}_{H^1(\OL)}^2+\norm{\uNL}_{H^1(\ONL)}^2+\delta\norm{\uG}_{L^2(\Gamma)}^2\leq C\norm{f}_{L^2(\Omega)}^2,
  \end{equation}
  where the constant $C$ is independent of $\delta$.
\end{theorem}

We can further prove the solution $(\uL,\uNL,\uG)$ converges to the solution of elliptic transmission problem as $\delta\rightarrow 0$ and give the convergence rate.
\begin{theorem}
\label{THEOREM:CONVERGENCE}
If $f\in H^1(\Omega)$, which ensures a $u\in H^3(\OL)\cap H^3(\ONL)$ solves (\ref{original system}), then the solution of (\ref{coupling system}) converges to $u$ with 
\begin{equation}
  \norm{u-\uL}_{H^1(\OL)}^2+\norm{u-\uNL}_{H^1(\ONL)}^2+\delta\norm{{\pd{u}{\nn}}^+-\uG}_{L^2(\Gamma)}^2\leq C\delta^2\norm{f}_{H^1(\Omega)}^2,
\end{equation}
where the constant $C$ is independent of $\delta$.
\end{theorem}

\section{Proof of well-posedness (Theorem \ref{THEOREM:WELL-POSEDNESS})}
\label{sec:well-posedness}
In this section, we present the proof of Theorem \ref{THEOREM:WELL-POSEDNESS}, i.e. the well-posedness of our local-nonlocal coupling method. 
The existence and uniqueness can be proved by verifying a bilinear form, which will be constructed later, satisfies Lax-Milgram theorem. 
Based on the coercivity of the bilinear form, the estimation (\ref{well-posedness estimation}) can be derived.

We first solve $u_\Gamma$ from the last equation of (\ref{coupling system}) and express $\uG$ with $\uL$ and $\uNL$, that is 
\begin{equation}
  \label{uGamma}
  \uG(\uL,\uNL)(\xx)=\dfrac{1}{\zeta_\delta(\xx)}(\uL(\xx)-\bbuNL(\xx)).
\end{equation} 
Now we can define a bilinear form $B: \tilde{H}\times\tilde{H}\rightarrow \RR$, with 
\begin{align}
  &B[\uL,\uNL;\vL,\vNL]\nonumber\\
	=&\lambda_1\int_{\OL}\nabla \uL(\xx)\cdot\nabla \vL(\xx)\d\xx+\lambda_2\int_{\Gamma}\uG(\uL,\uNL)(\xx)\vL(\xx)\d S_\xx\nonumber\\
	&+\dfrac{\lambda_2}{2\delta^2}\int_{\ONL}\int_{\ONL}\Rd (\uNL(\xx)-\uNL(\yy))(\vNL(\xx)-\vNL(\yy))\d\xx\d\yy\nonumber\\
	&\qquad\hspace{2.5cm} -\lambda_2\int_{\ONL}\vNL(\xx)\int_{\Gamma}\bRd \dfrac{\uG(\uL,\uNL)(\yy)}{\bar{w}_\delta(\yy)}\d S_\yy\d\xx,\label{bilinear form}
\end{align}
where 
\begin{equation*}
  \tilde{H}=\left\{(u_1,u_2):u_1\in H^1(\OL), u_2\in L^2(\ONL),\int_{\OL}u_1(\xx)\d\xx+\int_{\ONL}u_2(\xx)\d\xx=0\right\}
\end{equation*}
with norm 
\begin{equation*}
  \norm{(u_1,u_2)}^2_{\tilde{H}}=\norm{u_1}^2_{H^1(\OL)}+\norm{u_2}^2_{L^2(\ONL)}.
\end{equation*}

Notice that bilinear form $B$ is the summation of two parts. The first part is obtained by multiplying the left-hand side of the Poisson equation in (\ref{coupling system}) by $\vL$ and integrating by parts with the inner normal derivative $\uG(\uL,\uNL)$.
The another one is the $L^2$ inner product of $\vNL$ and the left-hand side of the second equation in (\ref{coupling system}).

In order to get the existence and uniqueness of the solution to our coupling model, we should first state the following theorem.
\begin{theorem}
  \label{theorem: weak existence and uniqueness}
  For $f_1(\xx)\in L^2(\OL)$ and $f_2\in L^2(\ONL)$, there exists a uniqueness pair $(\uL,\uNL)\in \tilde{H}$ such that 
  \[B[\uL,\uNL;\vL,\vNL]=(f_1,f_2;\vL,\vNL),\quad\forall (\vL,\vNL)\in \tilde{H}.\]
  Here 
  $$(f_1,f_2;\vL,\vNL)=\int_{\OL}f_1(\xx)\vL(\xx)\d\xx+\int_{\ONL}f_2(\xx)\vNL(\xx)\d\xx.$$
\end{theorem}
This theorem can be proved by Lax-Milgram theorem. That is to verify the continuity and coercivity of the bilinear form $B$.
\begin{proposition}
  \label{prop:Continuity}
  (Continuity)
  For any $(\uL,\uNL), (\vL,\vNL)\in \tilde{H}$, we have the following estimation
  \begin{align*}
    B[\uL,\uNL;\vL,\vNL]
\leq\dfrac{C}{\delta^2}\norm{(\uL,\uNL)}_{\tilde{H}}\norm{(\vL,\vNL)}_{\tilde{H}},
  \end{align*}
  where $C$ is independent of $\delta$.
\end{proposition}
\begin{proof}
If we can prove 
\begin{align}
\label{eq:continuous}
    B[\uL,\uNL;\vL,\vNL]
\leq\dfrac{C}{\delta^2}\left(\norm{\uL}_{H^1(\OL)}+\norm{\uNL}_{L^2(\ONL)}\right)\left(\norm{\vL}_{H^1(\OL)}+\norm{\vNL}_{L^2(\ONL)}\right),
  \end{align}
the continuity is then a corollary.
  
  There are four terms in (\ref{bilinear form}). For the first term, 
  \begin{align}
    \lambda_1\int_{\OL} \nabla \uL(\xx)\cdot \nabla \uNL(\xx)\d\xx&\leq \lambda_1\norm{\nabla \uL}_{L^2(\OL)}\norm{\nabla \vL}_{L^2(\OL)}
    \leq \lambda_1\norm{\uL}_{H^1(\OL)}\norm{\vL}_{H^1(\OL)}.\label{Continuity-1}
  \end{align}
  The interface function $\uG(\uL,\uNL)$ defined in (\ref{uGamma}) appears in the second and fourth terms of (\ref{bilinear form}). Since $C_1\delta\leq \zeta_\delta(\xx)\leq C_2\delta$, $\bwd(\xx)\geq C$ when $\xx\in\Gamma$, and 
  \begin{align*}
    \norm{\bbuNL}_{L^2(\Gamma)}^2
    =& \int_\Gamma\dfrac{1}{\bwd^2(\xx)}\left(\int_\ONL\bRd \uNL(\yy)\d\yy\right)^2\d S_\xx\\
    \leq& C\int_\Gamma\left(\int_\ONL\bRd\d\yy\right)\left(\int_\ONL\bRd \uNL^2(\yy)\d\yy\right)\d S_\xx\\
    \leq& C\int_\ONL \uNL^2(\yy)\int_\Gamma\bRd \d S_\xx\d\yy\\
    \leq& \frac{C}{\delta}\norm{\uNL}^2_{L^2(\ONL)},
  \end{align*}
  we can estimate the second term 
  \begin{align}
    &\lambda_2\int_{\Gamma} \uG(\uL,\uNL)(\xx)v(\xx)\d S_\xx\notag\\
    \leq&\dfrac{C}{\delta}\left(\int_{\Gamma}\left|\uL(\xx)\right|\left|\vL(\xx)\right|\d\xx+\int_{\Gamma}\left|\bbuNL(\xx)\right|\left|\vL(\xx)\right|\d\xx\right)\notag\\
    \leq&\dfrac{C}{\delta}\norm{\uL}_{L^2(\Gamma)}\norm{\vL}_{L^2(\Gamma)}+\dfrac{C}{\delta}\norm{\bbuNL}_{L^2(\Gamma)}\norm{\vL}_{L^2(\Gamma)}\notag\\
    \leq&\dfrac{C}{\delta}\norm{\uL}_{H^1(\OL)}\norm{\vL}_{H^1(\OL)}+\dfrac{C}{\delta^{3/2}}\norm{\uNL}_{L^2(\ONL)}\norm{\vL}_{H^1(\OL)}.\label{Continuity-2}
  \end{align}
  Similarly, the fourth term can be estimated as 
  \begin{align}
    &\lambda_2\int_{\ONL}\vNL(\xx)\int_{\Gamma}\bRd \dfrac{\uG(\uL,\uNL)(\yy)}{\bar{w}_\delta(\yy)}\d S_\yy\d\xx\notag\\
    \leq&C\int_{\ONL}\left|\vNL(\xx)\right|\int_{\Gamma}\bRd\left|\uL(\yy)\right|\d S_\yy\d\xx+C\int_{\ONL}\left|\vNL(\xx)\right|\int_{\Gamma}\bRd\left|\bbuNL(\yy)\right|\d S_\yy\d\xx\notag\\
    \leq& C\norm{\vNL}_{L^2(\ONL)}\left(\int_{\ONL}\left(\int_{\Gamma}\bRd\d S_\yy\right)\left(\int_{\Gamma}\bRd \uL^2(\yy)\d S_\yy\right)\d\xx\right)^{\frac{1}{2}}\notag\\
    &\ +C\norm{\vNL}_{L^2(\ONL)}\left(\int_{\ONL}\left(\int_{\Gamma}\bRd\d S_\yy\right)\left(\int_{\Gamma}\bRd \bbuNL^2(\yy)\d S_\yy\right)\d\xx\right)^{\frac{1}{2}}\notag\\
    \leq& \dfrac{C}{\sqrt{\delta}}\norm{\vNL}_{L^2(\ONL)}\norm{\uL}_{L^2(\Gamma)}+\dfrac{C}{\sqrt{\delta}}\norm{\vNL}_{L^2(\ONL)}\norm{\bbuNL}_{L^2(\Gamma)}\notag\\
    \leq& \dfrac{C}{\sqrt{\delta}}\norm{\uL}_{H^1(\OL)}\norm{\vNL}_{L^2(\ONL)}+\dfrac{C}{\delta}\norm{\uNL}_{L^2(\ONL)}\norm{\vNL}_{L^2(\ONL)}\label{Continuity-4}
  \end{align}

 What left in (\ref{bilinear form}) is the third term, we can directly estimate it by a similar way as above.
  \begin{align}
    &\dfrac{\lambda_2}{2\delta^2}\int_{\ONL}\int_{\ONL}\Rd (\uNL(\xx)-\uNL(\yy))(\vNL(\xx)-\vNL(\yy))\d\xx\d\yy\notag\\
  =&\dfrac{\lambda_2}{\delta^2}\int_{\ONL}\int_{\ONL}\Rd\left(\uNL(\xx)\vNL(\xx)-\uNL(\xx)\vNL(\yy)\right)\d\xx\d\yy\notag\\
  \leq& \dfrac{\lambda_2}{\delta^2}\int_{\ONL}\left|\uNL(\xx)\right|\left|\vNL(\xx)\right|\int_\ONL\bRd\d\yy\d\xx
+\dfrac{\lambda_2}{\delta^2}\int_{\ONL}\left|\uNL(\xx)\right|\int_{\ONL}\bRd\left|\vNL(\yy)\right|\d\yy\d\xx\notag\\
  \leq&\dfrac{C}{\delta^2}\norm{\uNL}_{L^2(\ONL)}\norm{\vNL}_{L^2(\ONL)}
+\dfrac{C}{\delta^2}\norm{\uNL}_{L^2(\ONL)}\left(\int_{\ONL}\left(\int_{\ONL}\bRd\left|\vNL(\yy)\right|\d\yy\right)^2\d\xx\right)^{\frac{1}{2}}\notag\\
  \leq& \dfrac{C}{\delta^2}\norm{\uNL}_{L^2(\ONL)}\norm{\vNL}_{L^2(\ONL)}.\label{Continuity-3}
  \end{align}
  Combining (\ref{Continuity-1})(\ref{Continuity-2})(\ref{Continuity-3}) and (\ref{Continuity-4}), we get (\ref{eq:continuous}), which means the estimation in Proposition \ref{prop:Continuity} is proved.
\end{proof}
\begin{proposition}
  \label{prop:Coercivity}
  (Coercivity)
  For a function pair $(\uL,\uNL)\in \tilde{H}$, there exists a constant $C$ such that  
  \begin{align*}
    B[\uL,\uNL;\uL,\uNL]\geq C\left(\norm{\uL}_{H^1(\OL)}^2+\norm{\uNL}_{L^2(\ONL)}^2\right).
  \end{align*}
  Here $C$ is independent of $\delta$.
\end{proposition}
The proof of coercivity is more involved. We need three lemmas to derive the above proposition. 
\begin{lemma}
  \label{lemma-1}
	For arbitrary $\uNL\in L^2(\ONL)$ and $\bbuNL$ defined as in (\ref{notation1}), there exists two positive constants $C_1, C_2$ such that 
	\begin{equation*}
		\dfrac{1}{2\delta^2}\int_{\ONL}\int_{\ONL}\Rd (\uNL(\xx)-\uNL(\yy))^2\d\xx\d\yy\geq C_1\norm{\nabla \bbuNL}_{L^2(\ONL)}^2,
	\end{equation*}
  and 
	\begin{equation*}
		\norm{\bbuNL-\uNL}_{L^2(\ONL)}^2 \leq C_2 \int_{\ONL}\int_{\ONL}\Rd (\uNL(\xx)-\uNL(\yy))^2\d\xx\d\yy.
	\end{equation*}
\end{lemma}
The first inequality is a classical result in nonlocal analysis, which can be found in \cite{shi2017convergence} and the second inequality was proved in \cite{wang2023nonlocal}. 
\begin{lemma}
  \label{lemma-2}
  There exists a constant $C>0$ depending only on the region, such that 
  \begin{align*}
    &\norm{\uL}_{L^2(\OL)}^2+\norm{\bbuNL}_{L^2(\ONL)}^2\\
    \leq& C\biggl[\norm{\uL-\bbuNL}_{L^2(\Gamma)}+\left(\norm{\nabla\uL}_{L^2(\OL)}^2+\norm{\nabla\bbuNL}_{L^2(\ONL)}^2\right)\\
&+\left(\int_{\OL}\uL(\xx)\d\xx+\int_{\ONL}\bbuNL(\xx)\d\xx\right)^2\biggr]
  \end{align*}
  for arbitrary $\uL\in H^1(\OL)$ and $\bbuNL\in H^1(\ONL)$.
\end{lemma}
This lemma is a Poincar$\mathrm{\acute{e}}$-type inequality, which can be proved by contradiction like the classical Poincar$\mathrm{\acute{e}}$'s inequality.  The proof can be found in \cite{zhang2021nonlocal}.
\begin{lemma}
  \label{lemma-3}
  For $(\uL,\uNL)\in \tilde{H}$ and $\bbuNL$ defined in (\ref{notation1}), there exists a constant $C>0$ independent of $\delta$ such that
  \begin{align}
  \label{Lemma equ4}
  \left(\int_{\OL}\uL(\xx)\d\xx+\int_{\ONL}\bbuNL(\xx)\d\xx\right)^2
  \leq C\int_{\ONL}\int_{\ONL}\Rd (\uNL(\xx)-\uNL(\yy))^2\d\xx\d\yy.
  \end{align}
\end{lemma}
\begin{proof}
  Since $(\uL,\uNL)\in \tilde{H}$, we have 
  \begin{align*}
    \int_\OL \uL(\xx)\d\xx+\int_\ONL \uNL(\xx)\d\xx=0.
  \end{align*}
  Now we can get
  \begin{align*}
    &\left(\int_{\OL}\uL(\xx)\d\xx+\int_{\ONL}\bbuNL(\xx)\d\xx\right)^2\\
  =&\left(\int_{\ONL}\bbuNL(\xx)\d\xx-\int_{\ONL}\uNL(\xx)\d\xx\right)^2\\
  \leq&C\int_{\ONL}(\bbuNL(\xx)-\uNL(\xx))^2\d\xx\\
  \leq&C\int_{\ONL}\int_{\ONL}\Rd (\uNL(\xx)-\uNL(\yy))^2\d\xx\d\yy,
  \end{align*}
  where the last inequality is ensured by Lemma {\ref{lemma-1}}.
\end{proof}

With these lemmas, we are ready to prove the coercivity. By a simple calculation,
\begin{align*}
	&B[\uL,\uNL;\uL,\uNL]\\
	=&\lambda_1\int_{\OL}\left|\nabla \uL(\xx)\right|^2\d\xx+\lambda_2\int_{\Gamma}\uG(\uL,\uNL)(\xx)\left(\uL(\xx)-\bbuNL(\xx)\right)\d S_\xx\\
	&\quad+\dfrac{\lambda_2}{2\delta^2}\int_{\ONL}\Rd (\uNL(\xx)-\uNL(\yy))^2\d\xx\d\yy\\
	=&\lambda_1\int_{\OL}\left|\nabla \uL(\xx)\right|^2\d\xx+\lambda_2\int_{\Gamma}\zeta_\delta(\xx)u_\Gamma^2(\uL,\uNL)(\xx)S_\xx\\
  &\hspace{3cm}+\dfrac{\lambda_2}{2\delta^2}\int_{\ONL}\Rd (\uNL(\xx)-\uNL(\yy))^2\d\xx\d\yy.
\end{align*}
Since $C_1\delta\leq \zeta_\delta(\xx)\leq C_2\delta$ when $\xx\in\Gamma$, the second term above can have a lower bound 
\begin{align}
\int_{\Gamma}\zeta_\delta(\xx)\uG^2(\uL,\uNL)(\xx)&=\int_{\Gamma}\dfrac{1}{\zeta_\delta(\xx)}\left(\uL(\xx)-\bbuNL(\xx)\right)^2\d S_\xx\nonumber\\
	&\geq\dfrac{C}{\delta}\int_{\Gamma}\left(\uL(\xx)-\bbuNL(\xx)\right)^2\d S_\xx\notag\\
	&\geq C\norm{\uL-\bbuNL}_{L^2(\Gamma)}^2.\label{boundary diff}
\end{align} 

Now we have the following estimation
\begin{align}
  &\norm{\uL}_{H^1(\OL)}^2+\norm{\uNL}_{L^2(\ONL)}^2\notag\\
  \leq&C\left(\norm{\uL}_{L^2(\OL)}^2+\norm{\bbuNL(\xx)}_{L^2(\ONL)}^2\right)+C\norm{\uNL-\bbuNL}_{L^2(\ONL)}^2+\norm{\nabla \uL}_{L^2(\OL)}^2\notag\\
  \leq&C\biggl[\left(\norm{\nabla\uL}_{L^2(\OL)}^2+\norm{\nabla\bbuNL}_{L^2(\ONL)}^2\right)+\left(\int_{\OL}\uL(\xx)\d\xx+\int_{\ONL}\bbuNL(\xx)\d\xx\right)^2\notag\\
  &+\norm{\uL-\bbuNL}_{L^2(\Gamma)}^2\biggr]+C \int_{\ONL}\int_{\ONL}\Rd (\uNL(\xx)-\uNL(\yy))^2\d\xx\d\yy+\norm{\nabla \uL}_{L^2(\OL)}^2\notag\\
  \leq&C\int_{\Gamma}\zeta_\delta(\xx)\uG^2(\uL,\uNL)(\xx)+C\int_{\ONL}\left|\nabla \uL(\xx)\right|^2\d\xx\notag\\
 &\hspace{3cm}+\dfrac{C}{2\delta^2}\int_{\ONL}\int_{\ONL}\Rd (\uNL(\xx)-\uNL(\yy))^2\d\xx\d\yy\notag\\
  \leq&C B[\uL,\uNL;\uL,\uNL].\label{coercivity proof}
\end{align}
This is exactly the coercivity.

Since the conditions of Lax-Milgram theorem have been verified, we can conclude Theorem \ref{theorem: weak existence and uniqueness}.

However, Theorem \ref{theorem: weak existence and uniqueness} is not equivalent to the existence and uniqueness of the solution to our local-nonlocal coupling model. 
To bridge the gap, we should further illustrate when $(f_1,f_2)$ is specified as $(f_L,f_{NL})$, the solution $(\uL,\uNL)$ in Theorem \ref{theorem: weak existence and uniqueness} satisfying 
\begin{align*}
  B[\uL,\uNL;\vL,\vNL]=(f_L,f_{NL};\vL,\vNL),\quad \forall \vL\in H^1(\OL),\forall \vNL\in L^2(\ONL).
 \end{align*}
In fact, for $\vL\in H^1(\OL)$, $\vNL\in L^2(\ONL)$ and a constant 
\[c=\dps\dfrac{1}{\left|\Omega\right|}\left(\int_{\OL}\vL(\xx)\d\xx+\int_{\ONL}\vNL(\xx)\d\xx\right),\] 
take $\bar{v}_{L}(\xx)=\vL(\xx)-c$ and $\bar{v}_{NL}(\xx)=\vNL(\xx)-c$, then $(\bar{v}_{L},\bar{v}_{NL})\in \tilde{H}.$  Now 
\begin{align*}
(f_L,f_{NL};\vL,\vNL)&=(f_L, f_{NL};\bar{v}_L+c,\bar{v}_{NL}+c)\\
&=c\left(\int_{\OL}f_L(\xx)\d\xx+\int_{\ONL}f_{NL}(\xx)\d\xx\right)+(f_L, f_{NL};\bar{v}_L,\bar{v}_{NL})\\
&=(f_L, f_{NL};\bar{v}_L,\bar{v}_{NL}),
\end{align*} 
and 
\begin{align*}
	B[\uL,\uNL;\vL,\vNL]&= B[\uL,\uNL; \bar{v}_L+c,\bar{v}_{NL}+c]\\
	&=B[\uL,\uNL; \bar{v}_L,\bar{v}_{NL}]+c\int_\Gamma \uG(\uL,\uNL)(\xx)\d S_\xx\\
	&\hspace{2cm}-c\int_{\Gamma} \uG(\uL,\uNL)(\yy)\dfrac{1}{\bar{w}_\delta(\yy)}\int_{\ONL}\bRd\d\xx\d S_\yy\\
	&=B[\uL,\uNL; \bar{v}_L,\bar{v}_{NL}].
\end{align*}
Because $B[\uL,\uNL;\bar{v}_L,\bar{v}_{NL}]=(f_L, f_{NL};\bar{v}_L,\bar{v}_{NL})$, this equation holds true for $\vL\in H^1(\OL)$, $\vNL\in L^2(\ONL)$.   
If we take $\vNL=0$, we can get for any $v_L\in H^1(\ONL)$
\begin{align}
  &\lambda_1\int_{\OL}\nabla u_L(\xx)\cdot\nabla \vL(\xx)\d\xx+ \lambda_2 \int_{\Gamma}\uG(\uL,\uNL)(\xx)\vL(\xx)\d\xx=\int_{\OL}f_L(\xx)\vL(\xx)\d\xx.\label{weak solution}
\end{align} 
In addition, if $\vL=0$ and $\vNL$ is taken as  
\begin{align*}
  \vNL (\xx)&=\dfrac{\lambda_2}{\delta^2}\int_{\ONL}\Rd (\uNL(\xx)-\uNL(\yy))\d\xx\notag\\
  &\hspace{3cm}-\lambda_2\int_{\Gamma}\bRd \dfrac{u_\Gamma(\uL,\uNL)(\yy)}{\bar{w}_\delta(\yy)}\d S_\yy-f_{NL}(\xx),
\end{align*}
we can obtain when $x\in\ONL$,
\begin{align}
  \dfrac{\lambda_2}{\delta^2}\int_{\ONL}\Rd (\uNL(\xx)-\uNL(\yy))\d\xx-\lambda_2\int_{\Gamma}\bRd \dfrac{u_\Gamma(\uL,\uNL)(\yy)}{\bar{w}_\delta(\yy)}\d S_\yy
  =f_{NL}(\xx).\label{coupling eq2}
\end{align}
Since $\uG(\uL,\uNL)$ is determined by $\uL$ and $\uNL$ according to the third equation in (\ref{coupling system}), taking a function $\uG(\xx)=\uG(\uL,\uNL)(\xx),\ \xx\in \Gamma$, $(\uL,\uNL,\uG)$ satisfies the third 
equation in (\ref{coupling system}). Substituting $\uG$ into the two results above, we can find (\ref{coupling eq2}) is exactly the second equation in (\ref{coupling system}) and (\ref{weak solution}) implies $\vL$ is the 
standard weak solution of the local part in (\ref{coupling system}). Here we have proved the existence and uniqueness of the solution to our model.

To complete the proof of Theorem \ref{THEOREM:WELL-POSEDNESS}, we define a new bilinear form $\hat B:\hat{H}\times \hat{H}\rightarrow \RR$, 
\begin{align*}
  &\hat{B}[\uL,\uNL,\uG;\vL,\vNL,\vG]\\
	=&\lambda_1\int_{\OL}\nabla \uL(\xx)\cdot\nabla \vL(\xx)\d\xx+\lambda_2\int_{\Gamma}\uG(\xx)\vL(\xx)\d S_\xx\\
	&+\dfrac{\lambda_2}{2\delta^2}\int_{\ONL}\Rd (\uNL(\xx)-\uNL(\yy))(\vNL(\xx)-\vNL(\yy))\d\xx\\
  &-\lambda_2\int_{\ONL}\vNL(\xx)\int_{\Gamma}\bRd \dfrac{\uG(\yy)}{\bar{w}_\delta(\yy)}\d S_\yy\d\xx\\
	&-\lambda_2\int_\Gamma \vG(\xx)\left(u_L(\xx)-\dfrac{1}{\bar{w}_\delta(\xx)}\int_{\ONL}\bRd u_{NL}(\yy)\d\yy\right)\d S_\xx+\lambda_2\int_{\Gamma}\zeta_{\delta}(\xx)u_{\Gamma}(\xx)v_\Gamma(\xx)\d S_\xx.
\end{align*}

A simple calculation gives 
\begin{align*}
  &\hat{B}[\uL,\uNL,\uG;\uL,\uNL,\uG]\\
	=&\lambda_1\int_{\OL}\left|\nabla \uL(\xx)\right|^2\d\xx+\lambda_2\int_\Gamma \zeta_\delta(\xx)\uG^2(\xx)\d S_\xx\\
  &+\dfrac{\lambda_2}{2\delta^2}\int_{\ONL}\int_{\ONL}\bRd(\uNL(\xx)-\uNL(\yy))^2\d\xx\d\yy
\end{align*}

Under the precondition $(\uL,\uNL,\uG)\in \hat{H}$ is the solution of our model, we can derive the following estimation
\begin{equation}
  \label{triple coercivity}
  \hat{B}[\uL,\uNL,\uG;\uL,\uNL,\uG]\geq C\left(\norm{\uL}_{H^1(\OL)}^2+\norm{\uNL}_{L^2(\ONL)}^2+\delta\norm{\uG}_{L^2(\Gamma)}^2\right)
\end{equation}
via a same way in the preceding proof of coercivity.
\begin{remark}
  \label{remark: bilinear form}
  It is notable (\ref{triple coercivity}) only holds when $(\uL,\uNL,\uG)$ solves (\ref{coupling system}), which means we should not use the bilinear form $\hat{B}$ to 
  prove the existence and uniqueness. Otherwise, these three functions are independent of each other, which results in the coupling condition on interface are violated. 
  The reason we solve $\uG$ from (\ref{coupling system}) in advance to define $B$ is to ascertain the relation between $\uG$ and $(\uL,\uNL)$, which is crucial to control $\norm{\uL-\bbuNL}_{L^2(\Gamma)}^2$
  in (\ref{boundary diff}) and (\ref{coercivity proof}).
\end{remark}

Next, we further prove $\uNL\in H^1(\ONL)$ and estimate its $H^1$ norm. In fact, from (\ref{coupling system}), we can get 
\begin{align*}
  \uNL(\xx)=\buNL(\xx)+\dfrac{\delta^2}{w_\delta(\xx)}\int_{\Gamma} \bRd \dfrac{\uG(\yy)}{\bar{w}_\delta(\yy)}\d S_\yy+\dfrac{\delta^2}{\lambda_2w_\delta(\xx)}f_{NL}(\xx),
\end{align*}
where $\buNL(\xx)$ is defined in (\ref{notation1}). Similar to Lemma \ref{lemma-1}, the following estimation also holds.
\begin{align}
  \norm{\nabla\buNL}_{L^2(\ONL)}^2\leq \dfrac{C}{2\delta^2}\int_{\ONL}\int_{\ONL}\Rd (\uNL(\xx)-\uNL(\yy))^2\d\xx\d\yy.\label{nabla baru}
\end{align}
Meanwhile, 
\begin{align*}
  &\nabla\left(\dfrac{\delta^2}{w_\delta(\xx)}\int_{\Gamma} \bRd \dfrac{\uG(\yy)}{\bar{w}_\delta(\yy)}\d S_\yy\right)\\
  =&\frac{\delta^2}{\wdd(\xx)}\int_{\Gamma}\nabla_\xx\bRd\dfrac{\uG(\yy)}{\bar{w}_\delta(\yy)}\d S_\yy-\frac{\delta^2\nabla \wdd(\xx)}{\wdd^2(\xx)}\int_{\Gamma} \bRd \dfrac{\uG(\yy)}{\bar{w}_\delta(\yy)}\d S_\yy\\
  \overset{d}{=}&I_1(\xx)+I_2(\xx).
\end{align*}
We can estimate 
\begin{align}
  \norm{I_1}_{L^2(\ONL)}^2
=&\delta^4\int_{\ONL}\dfrac{1}{\wdd^2(\xx)}\left(\int_{\Gamma}\Rd\frac{\xx-\yy}{2\delta^2}\dfrac{\uG(\yy)}{\bar{w}_\delta(\yy)}\d S_\yy\right)^2\d \xx\notag\\
\leq&C\delta^2\int_{\ONL}\left(\int_{\Gamma}\Rd\d S_\yy\right)\left(\int_{\Gamma}\Rd \uG^2(\yy)\d S_\yy\right)\d\xx\notag\\
\leq&C\delta\int_{\Gamma}\uG^2(\yy)\left(\int_{\ONL}\Rd\d \xx\right)\d S_\yy 
\leq C\int_\Gamma\zeta_\delta(\xx)\uG^2(\xx)\d S_\xx,\label{I1}
\end{align}
and 
\begin{align}
  \norm{I_2}_{L^2(\ONL)}^2
  =&\delta^4\int_{\ONL}\dfrac{1}{\wdd^4(\xx)}\left(\int_{\Gamma}\bRd\dfrac{\uG(\yy)}{\bar{w}_\delta(\yy)}\d S_\yy\right)^2\left(\int_{\ONL}R_\delta'(\xx,\yy)\frac{\xx-\yy}{2\delta^2}\d\yy\right)^2\d\xx\notag\\
  \leq&C\delta^2\int_{\ONL}\left(\int_{\Gamma}\bRd\d S_\yy\right)\left(\int_{\Gamma}\bRd \uG^2(\yy)\d S_\yy\right)\d\xx\notag\\
  \leq&C\delta\int_{\Gamma}\uG^2(\yy)\left(\int_{\ONL}\bRd\d \xx\right)\d S_\yy
\leq C\int_\Gamma\zeta_\delta(\xx)\uG^2(\xx)\d S_\xx.\label{I2}
\end{align}
In addition, 
\begin{align*}
  &\nabla\left(\dfrac{\delta^2}{\lambda_2w_\delta(\xx)}f_{NL}(\xx)\right)\\
  &=\dfrac{\delta^2}{\lambda_2\wdd(\xx)} \int_{\ONL}\nabla_\xx\bRd f(\yy)\d\yy+\dfrac{\delta^2\nabla \wdd(\xx)}{\lambda_2\wdd^2(\xx)}\left(\int_{\ONL}\bRd f(\yy)\d\yy+\bar{f}\right)\\
  &\overset{d}{=}J_1(\xx)+J_2(\xx).
\end{align*}
$J_1(\xx)$ can be estimated as 
\begin{align*}
  \norm{J_1}_{L^2(\ONL)}^2
=&\dfrac{\delta^4}{\lambda_2}\int_{\ONL}\dfrac{1}{\wdd^2(\xx)}\left(\int_{\ONL}\Rd \frac{\xx-\yy}{2\delta^2}f(\yy)\d\yy\right)^2\d\xx\\
\leq& C\delta^2\int_\ONL f^2(\yy)\left(\int_{\ONL}\Rd\d\xx\right)\d\yy\\
\leq& C\delta^2\norm{f}_{L^2(\Omega)}.
\end{align*}
With the help of Lemma \ref{LEMMA:F}, we have  
\begin{align*}
  \norm{J_2(\xx)}_{L^2(\ONL)}^2
  =&\dfrac{\delta^4}{\lambda_2}\int_{\ONL}\dfrac{1}{\wdd^4(\xx)}\left(\int_{\ONL}\bRd f(\yy)\d\yy+\bar{f}\right)^2\left(\int_{\ONL}R'_\delta(\xx,\yy)\dfrac{\xx-\yy}{2\delta^2}\right)^2\d\xx\\
  \leq&C\delta^2\int_{\ONL}\left(\int_{\ONL}\bRd f(\yy)\d\yy+\bar{f}\right)^2\d\xx\\
  \leq&C\delta^2\left|\bar{f}\right|^2+C\delta^2\int_{\ONL}\left(\int_{\ONL}\bRd f(\yy)\d\yy\right)^2\d\xx\\
  \leq&C\delta^2\norm{f}_{L^2(\ONL)}^2.
\end{align*}

Based on these estimations above, we can conclude that 
\begin{align}
  \norm{\nabla \uNL}_{L^2(\ONL)}^2&\leq C\delta^2\norm{f}_{L^2(\ONL)}^2+C\int_{\Gamma}\zeta_\delta(\xx)\uG(\xx)\d S_\xx\notag\\
  &\hspace{0.5cm}+ \dfrac{C}{2\delta^2}\int_\ONL\int_\ONL\Rd (\uNL(\xx)-\uNL(\yy))^2\d\xx\d\yy\label{grad estimation}.
\end{align}
Combine (\ref{grad estimation}) and (\ref{triple coercivity}), we get the following result 
\begin{align*}
  &\norm{\uL}_{H^1(\OL)}^2+\norm{\uNL}_{H^1(\ONL)}^2+\delta\norm{\uG}_{L^2(\Gamma)}^2\leq C\hat{B}[\uL,\uNL,\uG;\uL,\uNL,\uG]+C\delta^2\norm{f}_{L^2(\Omega)}^2.
\end{align*}
Furthermore, $(\uL,\uNL,\uG)$ solves (\ref{coupling system}) implies 
\begin{align*}
  &\hat{B}[\uL,\uNL,\uG;\uL,\uNL,\uG]\\
  =&(f_L,f_{NL},0;\uL,\uNL,\uG)\\
  =&\int_{\OL}f_L(\xx)\uL(\xx)\d\xx+\int_{\ONL}f_{NL}(\xx)\uNL(\xx)\d\xx\\
  \leq& C(\epsilon)\left(\norm{f_L}_{L^2(\OL)}^2+\norm{f_{NL}}_{L^2(\OL)}^2\right)+\epsilon\left(\norm{\uL}_{L^2(\OL)}^2+\norm{\uNL}_{L^2(\ONL)}^2\right)\\
  \leq& C(\epsilon)\norm{f}_{L^2(\Omega)}^2+\epsilon\left(\norm{\uL}_{H^1(\OL)}^2+\norm{\uNL}_{H^1(\ONL)}^2\right)
\end{align*}
Take $\epsilon$ small enough, we can get the following estimation 
\begin{align*}
  \norm{\uL}_{H^1(\OL)}^2+\norm{\uNL}_{H^1(\ONL)}^2+\delta\norm{\uG}_{L^2(\Gamma)}^2\leq C\norm{f}_{L^2(\Omega)}^2.
\end{align*}
Here we have completed the proof of Theorem \ref{THEOREM:WELL-POSEDNESS}.

\section{Proof of convergence(Theorem \ref{THEOREM:CONVERGENCE})}
\label{sec:convergence proof}
In this section, we present the proof of Theorem \ref{THEOREM:CONVERGENCE}. A system about the errors will be established, and the truncation errors will be analyzed.
We will also build a new estimation like the coercivity, which will help us derive the convergence result.
\subsection{Truncation error analysis}
Our local-nonlocal coupling model is modified from (\ref{PIM}) and (\ref{interfaceequ1}). Before analyzing the convergence order, we should be clear about the truncation errors. 
\begin{proposition}
  \label{PROP:TRUNCATION ERROR INTERFACE}
  Let $u\in H^1(\Omega)\cap H^3(\OL)\cap H^3(\ONL)$ solves (\ref{original system}), the truncation error 
  \begin{align*}
    r_\Gamma(\xx)=-\lambda_2\left(u(\xx)-\dfrac{1}{\bwd(\xx)}\int_{\ONL}\bRd u(\yy)\d\yy\right)+\lambda_2\zeta_\delta(\xx)\pd{u}{\nn}(\xx),\ \xx\in\Gamma
  \end{align*}
  has the following estimation 
  \begin{align*}
    \norm{r_\Gamma}_{L^2(\Gamma)}\leq C\delta^2 \norm{u}_{H^3(\ONL)}.
  \end{align*}
\end{proposition}
We put the proof of above proposition in Appendix \ref{appendix:proof2}.
\begin{proposition}
  \label{prop: truncation error nonlocal}
  Let $u\in H^1(\Omega)\cap H^3(\ONL)\cap H^3(\ONL)$ be the solution of (\ref{original system}), then the truncation error $r_{NL}$ is the summation of the following two functions
\begin{align*}
  r_{NL,1}(\xx)&=\dfrac{1}{\delta^2}\int_{\ONL}\Rd(u(\xx)-u(\yy))\d\yy-2\int_\Gamma \bRd\pd{u}{\nn}(\yy)\d S_\yy
-\int_{\ONL}\bRd f(\yy)\d\yy,\\
  r_{NL,2}(\xx)&=\int_{\Gamma}\bRd\pd{u}{\nn}(\yy)\dfrac{2\bar{w}_\delta(\yy)-1}{\bar{w}_\delta(\yy)}\d S_\yy.
\end{align*}
Furthermore, $r_{NL,1}(\xx)$ can be decomposed into $r_{NL,1}(\xx)=r_{in}(\xx)+r_{bl}(\xx)$ and 
\begin{alignat*}{3}
	\norm{r_{in}}^2_{L^2(\ONL)}&\leq C\delta^2 \norm{u}^2_{H^3(\ONL)},\quad &\norm{\nabla r_{in}}^2_{L^2(\ONL)}&\leq C \norm{u}^2_{H^3(\ONL)},\\
	\norm{r_{bl}}^2_{L^2(\ONL)}&\leq C\delta \norm{u}^2_{H^3(\ONL)},\quad&\norm{\nabla r_{bl}}^2_{L^2(\ONL)}&\leq \dfrac{C}{\delta} \norm{u}^2_{H^3(\ONL)},\\
	\norm{r_{NL,2}}^2_{L^2(\ONL)}&\leq C\delta \norm{u}^2_{H^3(\ONL)},\quad &\norm{\nabla r_{NL,2}}^2_{L^2(\ONL)}&\leq \dfrac{C}{\delta} \norm{u}^2_{H^3(\ONL)}.
\end{alignat*}
In particular, 
\[r_{bl}(\xx)=\int_{\Gamma}(x^j-y^j)n^l(\yy)\nabla^j\nabla^l u(\yy)\bRd\d S_\yy.\]
For $h\in H^1(\ONL)$, we have
\[\int_{\ONL} r_{bd}(\xx)h(\xx)\d\xx\leq C\delta \norm{u}_{H^3(\ONL)}\norm{h}_{H^1(\ONL)}.\]
\end{proposition}
Here we only prove the results about $r_{NL,2}$, the rest conclusions can be found in \cite{shi2017convergence}. 
To derive the two estimations in Lemma \ref{prop: truncation error nonlocal}, we need the following lemma.
\begin{lemma}
  \label{INTEGRAL ERROR}
  For the kernel $\bRd$, when $\delta$ is small enough, there exists a constant $C$ independent of $\delta$ such that 
  \begin{align*}
    \left|2\int_{\ONL}\bRd\d\yy-1\right|\leq C\delta,\quad\forall\xx\in\Gamma.
  \end{align*}
\end{lemma}
The proof of this lemma is put in Appendix \ref{appendix:proof3}. With this lemma, 
\begin{align*}
  \norm{r_{NL,2}}_{L^2(\ONL)}^2&=\int_{\ONL}\left(\int_{\Gamma}\bRd\pd{u}{\nn}(\yy)\dfrac{2\bar{w}_\delta(\yy)-1}{\bar{w}_\delta(\yy)}\d S_\yy\right)^2\d\xx\\
  &\leq C\delta^2\int_\ONL\left(\int_\Gamma\bRd\d S_\yy\right)\left(\int_\Gamma\bRd \left|\pd{u}{\nn}(\yy)\right|^2\d S_\yy\right)\d\xx\\
  &\leq C\delta\int_\Gamma \left|\pd{u}{\nn}(\yy)\right|^2\d S_\yy
  \leq C\delta \norm{u}_{H^3(\ONL)}^2
\end{align*}
and 
\begin{align*}
  \norm{\nabla r_{NL,2}}_{L^2(\ONL)}^2&=\int_{\ONL}\left(\int_{\Gamma}\Rd \dfrac{\xx-\yy}{2\delta^2}\pd{u}{\nn}(\yy)\dfrac{2\bar{w}_\delta(\yy)-1}{\bar{w}_\delta(\yy)}\right)^2\d S_\yy\\
  &\leq C\int_\ONL\left(\int_\Gamma\bRd\d S_\yy\right)\left(\int_\Gamma\bRd \left|\pd{u}{\nn}(\yy)\right|^2\d S_\yy\right)\d\xx\\
  &\leq \dfrac{C}{\delta}\int_\Gamma \left|\pd{u}{\nn}\right|^2\d S_\yy
  \leq \dfrac{C}{\delta}\norm{u}_{H^3(\ONL)}^2.
\end{align*}

\subsection{Convergence analysis}
Let $u\in H^1(\Omega)\cap H^3(\ONL)\cap H^3(\ONL)$ be the solution to the elliptic transmission problem (\ref{original system}) and $(\uL,\uNL,\uG)$ solves our local-nonlocal coupling system (\ref{coupling system}), we denote 
\begin{align*}
  \eL(\xx)=u(\xx)-\uL(\xx);\quad \eNL(\xx)=u(\xx)-\uNL(\xx);\quad \eG(\xx)={\pd{u}{\nn}}^+(\xx)-\uG(\xx).
\end{align*}
Now we can establish the following system about $(\eL,\eNL,\eG)$ by a simple calculation,
\begin{equation}
\label{error system}
\left\{
\begin{aligned}
  &\lambda_1\int_{\OL}\nabla \eL(\xx)\cdot\nabla v(\xx)\d\xx+\lambda_2\int_{\Gamma}e_\Gamma(\xx)v(\xx)\d S_\xx=0, \quad \forall \vL\in H^1(\OL);\\
		&\dfrac{\lambda_2}{\delta^2}\int_{\ONL}\Rd (\eNL(\xx)-\eNL(\yy))\d\xx-\lambda_2\int_{\Gamma}\bRd \dfrac{e_\Gamma(\yy)}{\bwd(\yy)}\d S_\yy=r_{NL}(\xx)-\bar{f}; \\
		&- \lambda_2\left(\eL(\xx)-\dfrac{1}{\bar{w}_\delta(\xx)}\int_{\ONL}\bRd\eNL(\yy)\d\yy\right)+\lambda_2\zeta_\delta(\xx)e_\Gamma(\xx)=r_{\Gamma}(\xx).
\end{aligned} 
\right.
\end{equation}
Here $r_{NL},r_{\Gamma}$ in the right-hand side is exactly the functions in Proposition \ref{prop: truncation error nonlocal} and Proposition \ref{PROP:TRUNCATION ERROR INTERFACE}. In addition, $\bar{f}$ is the constant introduced in Lemma \ref{LEMMA:F}.
As mentioned in Remark \ref{remark: bilinear form}, the coercivity  (\ref{triple coercivity}) of bilinear form $\hat{B}$ can not hold since $r_\Gamma\neq 0$ in (\ref{error system}).
However, we can build a similar estimation 
\begin{align}
  \norm{\eL}_{H^1(\ONL)}^2+\norm{\eNL}_{L^2(\ONL)}^2+&\delta\norm{\eG}_{L^2(\Gamma)}^2
  \leq C\hat{B}[\eL,\eNL,\eG;\eL,\eNL,\eG]+\frac{C}{\delta}\norm{r_\Gamma}_{L^2(\Gamma)}^2.\label{new coercivity}
\end{align}
In fact 
\begin{equation*}
	\begin{aligned}
		&\hat{B}[\eL,\eNL,\eG;\eL,\eNL,\eG]\\
		=&\int_{\OL}\left|\nabla \eL(\xx)\right|^2\d\xx+\dfrac{1}{2\delta^2}\int_{\ONL}\int_{\ONL}\bRd(\eNL(\xx)-\eNL(\yy))^2\d\xx\d\yy+\int_\Gamma \zeta_\delta(\xx)\eG^2(\xx)\d S_\xx.
	\end{aligned}
\end{equation*}
We only need to estimate the third term. Following the notations in (\ref{notation1}),
\begin{equation*}
  \eG(\xx)=\dfrac{1}{\zeta_\delta(\xx)}\left(\dfrac{r_\Gamma(\xx)}{\lambda_2}+(\eL(\xx)-\bar{\bar{e}}_{NL}(\xx))\right).
\end{equation*}
We can get
\begin{align*}
    &\int_\Gamma\zeta_\delta(\xx)\eG^2(\xx)\d S_\xx\\
    =&\int_\Gamma \dfrac{1}{\lambda_2^2\zeta_\delta(\xx)}r_\Gamma^2(\xx)\d S_\xx+\int_\Gamma\dfrac{1}{\zeta_\delta(\xx)}(\eL(\xx)-\bbeNL(\xx))^2\d S_\xx+\int_\Gamma \dfrac{2r_\Gamma(\xx)}{\lambda_2\zeta_\delta(\xx)}(\eL(\xx)-\bbeNL(\xx))\d S_\xx\\
    \geq&\int_\Gamma \dfrac{1}{\lambda_2^2\zeta_\delta(\xx)}r_\Gamma^2(\xx)\d S_\xx+\int_\Gamma\dfrac{1}{\zeta_\delta(\xx)}(\eL(\xx)-\bbeNL(\xx))^2\d S_\xx-2\int_\Gamma \dfrac{\sqrt{2}\left|r_\Gamma(\xx)\right|}{\lambda_2\sqrt{\zeta_\delta(\xx)}}\dfrac{\left|\eL(\xx)-\bbeNL(\xx)\right|}{\sqrt{2}\sqrt{\zeta_\delta(\xx)}}\d S_\xx\\
    \geq&-\int_\Gamma \dfrac{1}{\lambda_2^2\zeta_\delta(\xx)}r_\Gamma^2(\xx)\d S_\xx+\dfrac{1}{2}\int_\Gamma\dfrac{1}{\zeta_\delta(\xx)}(\eL(\xx)-\bbeNL(\xx))^2\d S_\xx\\
    \geq&-\dfrac{C}{\delta}\int_\Gamma r_\Gamma^2(\xx)\d S_\xx+C\int_\Gamma (\eL(\xx)-\bbeNL(\xx))^2\d S_\xx,
\end{align*}
which means 
\begin{equation}
  \label{new interface estimation}
	\norm{\eL-\bbeNL}_{L^2(\Gamma)}^2\leq \dfrac{C}{\delta}\norm{r_\Gamma}_{L^2(\Gamma)}^2+C\int_\Gamma \zeta_\delta(\xx)\eG^2(\xx)\d S_\xx.
\end{equation}
Then, noticing that 
\begin{align*}
  \int_{\OL}\eL(\xx)\d\xx+\int_{\ONL}\eNL(\xx)\d\xx=0,
\end{align*}
Lemma \ref{lemma-3} also holds for $(\eL,\eNL)$. Similar to (\ref{coercivity proof}), combine Lemma \ref{lemma-1}, Lemma \ref{lemma-2}, Lemma \ref{lemma-3} and (\ref{new interface estimation}), we have 
\begin{align}
  &\norm{\eL}_{H^1(\OL)}^2+\norm{\eNL}_{L^2(\ONL)}^2+\delta\norm{\eG}_{L^2(\Gamma)}^2\notag\\
  \leq&C\norm{\eNL-\bbeNL}_{L^2(\ONL)}^2+\norm{\nabla \eL}_{L^2(\OL)}^2+\delta\norm{\eG}_{L^2(\Gamma)}^2+C\left(\norm{\eL}_{L^2(\OL)}^2+\norm{\bbeNL(\xx)}_{L^2(\ONL)}^2\right)\notag\\
  \leq& C \int_{\ONL}\int_{\ONL}\Rd (\eNL(\xx)-\eNL(\yy))^2\d\xx\d\yy+\norm{\nabla \eL}_{L^2(\OL)}^2+\int_{\Gamma}\zeta_\delta(\xx)\eG^2(\xx)\d\xx\notag\\
  &+C\biggl[\left(\norm{\nabla\eL}_{L^2(\OL)}^2+\norm{\nabla\bbeNL}_{L^2(\ONL)}^2\right)+\left(\int_{\OL}\eL(\xx)\d\xx+\int_{\ONL}\bbeNL(\xx)\d\xx\right)^2+\norm{\eL-\bbeNL}_{L^2(\Gamma)}^2\biggr]\notag\\
  \leq&C\int_{\Gamma}\zeta_\delta(\xx)\eG^2(\xx)+\dfrac{C}{\delta}\norm{r_{\Gamma}}_{L^2(\Gamma)}^2+C\int_{\ONL}\left|\nabla \eL(\xx)\right|^2\d\xx\notag\\
  &\hspace{3cm}+\dfrac{C}{2\delta^2}\int_{\ONL}\int_{\ONL}\Rd (\eNL(\xx)-\eNL(\yy))^2\d\xx\d\yy\notag\\
  \leq&C \hat{B}[\eL,\eNL,\eG;\eL,\eNL,\eG]+\dfrac{C}{\delta}\norm{r_{\Gamma}}_{L^2(\Gamma)}^2.\label{semi-coercivity}
\end{align}

In addition, from (\ref{error system}) we can get 
\[\eNL(\xx)=\dfrac{\delta^2}{\lambda_2w_\delta(\xx)}(r_{NL}(\xx)-\bar{f})+\bar{e}_{NL}(\xx)+\dfrac{\delta^2}{w_\delta(\xx)}\int_{\Gamma}\bRd \dfrac{\eG(\yy)}{\bar{w}_\delta(\yy)}\d S_\yy\]
The gradient of second term and the third term can be estimated in a same way like (\ref{nabla baru}) (\ref{I1}) and (\ref{I2}). The results are 
\begin{align}
  \norm{\nabla\bar{e}_{NL}}_{L^2(\ONL)}^2\leq \dfrac{C}{\delta^2}\int_{\ONL}\int_{\ONL}\Rd (\uNL(\xx)-\uNL(\yy))^2\d\xx\d\yy\label{nabla second term}
\end{align}
and  
\begin{align}
  \norm{\nabla\left(\dfrac{\delta^2}{w_\delta(\xx)}\int_{\Gamma}\bRd \dfrac{\eG(\yy)}{\bar{w}_\delta(\yy)}\d S_\yy\right)}_{L^2(\ONL)}^2\leq C\int_{\Gamma}\zeta_\delta(\xx)\eG(\xx)^2\d S_\xx.\label{nabla third term}
\end{align}
We next estimate the gradient of the first term. 
\begin{align*}
  \nabla \left(\dfrac{\delta^2}{\lambda_2w_\delta(\xx)}\left(r_{NL}(\xx)-\bar{f}\right)\right)
&=\dfrac{\delta^2}{\lambda_2\wdd(\xx)}\nabla r_{NL}(\xx)-\dfrac{\delta^2\nabla \wdd(\xx)}{\lambda_2\wdd^2(\xx)}(r_{NL}(\xx)-\bar{f})
\\&\overset{d}{=}K_1(\xx)+K_2(\xx).
\end{align*}
Then we have 
\begin{align}
  \norm{K_1(\xx)}_{L^2(\ONL)}^2&=\dfrac{\delta^4}{\lambda_2^2}\int_{\ONL}\dfrac{1}{\wdd^2(\xx)}\left(\nabla r_{in}(\xx)+\nabla r_{bl}(\xx)+\nabla r_{NL,2}(\xx)\right)^2\d\xx\notag\\
  &\leq C\delta^4\norm{\nabla r_{in}}_{L^2(\ONL)}^2+C\delta^4\norm{\nabla r_{in}}_{L^2(\ONL)}^2+C\delta^4\norm{\nabla r_{NL,2}}_{L^2(\ONL)}^2\notag\\
  &\leq C\delta^3\norm{u}_{H^3(\ONL)}^2
  \leq C\delta^3\norm{f}_{H^1(\Omega)}^2\label{K1}
\end{align}
and 
\begin{align}
  \norm{K_2(\xx)}_{L^2(\ONL)}^2&=\dfrac{\delta^4}{\lambda_2^2}\int_{\ONL}\dfrac{1}{\wdd^2(\xx)}\left(r_{NL,2}(\xx)-\bar{f}\right)^2\left(\int_{\ONL}R_\delta'(\xx,\yy)\dfrac{\xx-\yy}{2\delta^2}\d\yy\right)^2\d\xx\notag\\
  &\leq C\delta^2\norm{r_{NL,2}}_{L^2(\ONL)}^2+C\delta^2\left|\bar{f}\right|^2\leq C\delta^3\norm{f}^2_{H^1(\Omega)}.\label{K2}
\end{align}
Combining (\ref{K1})(\ref{K2})(\ref{nabla second term})(\ref{nabla third term}) and (\ref{semi-coercivity}), we get 
\begin{align}
&\norm{\eL}_{H^1(\OL)}^2+\norm{\eNL}_{H^1(\ONL)}^2+\delta\norm{\eG}_{L^2(\Gamma)}^2\notag\\
  \leq& C \hat{B}[\eL,\eNL,\eG;\eL,\eNL,\eG]+\dfrac{C}{\delta}\norm{r_\Gamma}_{L^2(\Gamma)}^2+C\delta^3\norm{f}_{H^1(\Omega)}^2.\label{eq: H1aid}
\end{align}
Furthermore, from the three equations in (\ref{error system}), 
\begin{align*}
    &\hat{B}[\eL,\eNL,\eG;\eL,\eNL,\eG]\\
  =&\int_{\ONL}(r_{NL}(\xx)-\bar{f})\eNL(\xx)\d\xx+\int_{\Gamma}r_{\Gamma}(\xx)\eG(\xx)\d\xx\notag\\
  =&\int_{\ONL}r_{in}(\xx)\eNL(\xx)\d\xx+\int_{\ONL}r_{bl}(\xx)\eNL(\xx)\d\xx+\int_{\ONL}r_{NL,2}(\xx)\eNL(\xx)\d\xx\notag\\
  &\hspace{3cm}-\int_{\ONL}\bar{f}\eNL(\xx)\d\xx+\int_{\ONL}r_{\Gamma}(\xx)\eG(\xx)\d\xx
\end{align*}
With the truncation error analysis in Proposition \ref{PROP:TRUNCATION ERROR INTERFACE} and Proposition \ref{prop: truncation error nonlocal} as well as the  estimation of $\bar{f}$ in Lemma \ref{LEMMA:F}, the bilinear form can be estimated as follows,
\begin{align}
  &\hat{B}[\eL,\eNL,\eG;\eL,\eNL,\eG]\notag\\
  \leq& \norm{r_{in}}_{L^2(\ONL)}\norm{\eNL}_{L^2(\ONL)}+C\delta\norm{u}_{H^3(\ONL)}\norm{\eNL}_{H^1(\ONL)}\notag\\
  &\hspace{0.5cm}+C\left|\bar{f}\right|\norm{\eNL}_{L^2(\ONL)}+\dfrac{1}{\sqrt{\delta}}\|r_{\Gamma}\|_{L^2(\Gamma)}\sqrt{\delta}\norm{\eG}_{L^2(\Gamma)}+\int_{\ONL}r_{NL,2}(\xx)\eNL(\xx)\d\xx\notag\\
  \leq& C(\epsilon)\left(\norm{r_{in}}_{L^2(\ONL)}^2+\delta^2\norm{u}_{H^3(\ONL)}^2+\delta^2\norm{f}_{H^1(\Omega)}^2+\dfrac{1}{\delta}\norm{r_{\Gamma}}_{L^2(\Gamma)}^2\right)\notag\\
  &\hspace{3cm}+\epsilon\left(\norm{\eNL}_{L^2(\ONL)}^2+\delta\norm{\eG}_{L^2(\Gamma)}^2\right)+\int_{\ONL}r_{NL,2}(\xx)\eNL(\xx)\d\xx.\notag\\
  \leq& C(\epsilon)\left(\delta^2\norm{u}_{H^3(\ONL)}^2+\delta^2\norm{f}_{H^1(\Omega)}^2+\delta^3\norm{u}_{H^3(\ONL)}^2\right)\notag\\
  &\hspace{2cm}+\epsilon\left(\norm{\eNL}_{H^1(\ONL)}^2+\delta\norm{\eG}_{L^2(\Gamma)}^2\right)+\int_{\ONL}r_{NL,2}(\xx)\eNL(\xx)\d\xx.\label{rightside}
\end{align}
We can further estimate 
\begin{align}
&\int_{\ONL}r_{NL,2}(\xx)\eNL(\xx)\d\xx\notag\\
=&\int_{\ONL}e_{NL}(\xx)\int_{\Gamma}\bRd\pd{u}{\nn}(\yy)\dfrac{2\bar{w}_\delta(\yy)-1}{\bar{w}_\delta(\yy)}\d S_\yy\d\xx\notag\\
=&\int_{\Gamma}\pd{u}{\nn}(\yy)\dfrac{2\bar{w}_\delta(\yy)-1}{\bar{w}_\delta(\yy)}\int_{\ONL}\bRd \eNL(\xx)\d\xx\d S_\yy\notag\\
=&\int_{\Gamma}\pd{u}{\nn}(\yy)(2\bar{w}_\delta(\yy)-1)\bbeNL(\yy)\d S_\yy\notag\\
\leq& C\int_{\Gamma}\left|\pd{u}{\nn}(\yy)\right|^2(2\bar{w}_\delta(\yy)-1)^2\d S_\yy+\epsilon \int_{\Gamma}\bbeNL^2(\yy)\d S_\yy\notag\\
\leq& C\delta^2\norm{u}^2_{H^2(\ONL)}+C\epsilon \norm{\bbeNL}^2_{H^1(\ONL)}.\notag\\
\leq& C\delta^2\norm{u}^2_{H^2(\ONL)}+C\epsilon\norm{\eNL}_{L^2(\ONL)}^2+C\epsilon\dfrac{1}{\delta^2}\int_{\ONL}\int_{\ONL}\Rd (\eNL(\xx)-\eNL(\yy))^2\d\xx\d\yy\label{extra error}
\end{align}
Here, Lemma \ref{lemma-1} is used in the last inequality. Taking $\epsilon$ small enough, we can get the following estimation about $\hat{B}[\eL,\eNL,\eG;\eL,\eNL,\eG]$ from (\ref{rightside}) and (\ref{extra error}), 
\begin{align}
  &\hat{B}[\eL,\eNL,\eG;\eL,\eNL,\eG]\notag\\
  \leq& C\delta^2\left(\norm{u}_{H^3(\ONL)}^2+\norm{f}_{H^1(\Omega)}^2\right)+C\epsilon\left(\norm{\eNL}_{H^1(\ONL)}^2+\delta\norm{\eG}_{L^2(\Gamma)}^2\right)\notag\\
  \leq& C\delta^2\norm{f}_{H^1(\Omega)}^2+C\epsilon\left(\norm{\eNL}_{H^1(\ONL)}^2+\delta\norm{\eG}_{L^2(\Gamma)}^2\right).\label{bilinear estimation}
\end{align}
Combining (\ref{eq: H1aid})(\ref{bilinear estimation}) and Proposition \ref{PROP:TRUNCATION ERROR INTERFACE}, we derive the following estimation 
\begin{align*}
\norm{\eL}_{H^1(\OL)}^2+\norm{\eNL}_{H^1(\ONL)}^2+\delta\norm{\eG}_{L^2(\Gamma)}^2\leq C\delta^2\norm{f}_{H^1(\Omega)}^2,
\end{align*}
which is exactly the final convergence result.

\section{Discussion and Conclusion}
In this paper, we proposed a local-nonlocal coupling model. By establishing some proper bilinear forms, the well-posedness of our model and 
a convergence to elliptic transmission problem in $H^1$ norm with order $O(\delta)$ have been proved. In addition, we notice that our model have no truncation error in 
local part. Hence, improving the accuracy of the nonlocal part may prompt a second order local-nonlocal coupling method. In fact, some works like \cite{zhang2021nonlocal} and \cite{meng2023maximum} have presented 
some second order nonlocal models, which is useful to design new coupling model. We will investigate this interesting problem in the future work.

\appendix

\section{Proof of Lemma \ref{LEMMA:F}}
\label{appendix:proof1}
There are two estimations about the constant $\bar{f}$ in Lemma \ref{LEMMA:F}. The first one is simple. Because $\bar{f}$ is introduced for 
\begin{align*}
	\int_{\OL}f(\xx)\d\xx+\int_{\ONL}\left(\int_{\ONL}\bRd f(\yy)\d\yy+\bar{f}\right)\d\xx=0.
\end{align*}
Thus,
\begin{align*}
	\left|f\right|&\leq \dfrac{1}{\left|\ONL\right|}\left(\int_\OL\left|f(\xx)\right|\d\xx+\int_{\ONL}\int_{\ONL}\bRd \left|f(\yy)\right|\d\yy\d\xx\right)\\
	&\leq C\norm{f}_{L^2(\OL)}+C\left(\int_{\ONL}\left(\int_{\ONL}\bRd \left|f(\yy)\right|\d\yy\right)^2\d\xx\right)^{\frac{1}{2}}\\
	&\leq C\norm{f}_{L^2(\OL)}+C\norm{f}_{L^2(\ONL)}\\
	&\leq C\norm{f}_{L^2(\Omega)}.
\end{align*}
The second estimation is more involved. Elliptic transmission problem (\ref{original system}) implies the compatibility condition 
\begin{equation*}
	\int_{\Omega}f(\xx)\d\xx=\int_{\OL}f(\xx)\d\xx+\int_{\ONL}f(\xx)\d\xx=0,
\end{equation*}
which gives 
\begin{align*}
	\bar{f}&=\dfrac{1}{\left|\ONL\right|}\left(-\int_{\OL}f(\xx)\d\xx-\int_{\ONL}\int_{\ONL}\bRd f(\yy)\d\yy\d\xx\right)\\
	&=\dfrac{1}{\left|\ONL\right|}\left(\int_{\ONL}f(\xx)\d\xx-\int_{\ONL}f(\yy)\int_{\ONL}\bRd \d\xx\d\yy\right)\\
	&=\dfrac{1}{\left|\ONL\right|}\left(\int_{\ONL}f(\xx)\left(1-\int_{\ONL}\bRd \d\yy\right)\d\xx\right)
\end{align*}
Let us denote $\Gamma_{2\delta}=\left\{\xx\in \ONL \big|\ d(\xx,\Gamma)<2\delta\right\}$. Notice that 
\begin{align*}
	1-\int_{\ONL}\bRd\d\yy=0,\quad \forall \xx\in (\ONL\backslash\Gamma_{2\delta}),
\end{align*}
we can further get 
\begin{align*}
	\left|\bar{f}\right|&=\left|\dfrac{1}{\left|\ONL\right|}\left(\int_{\Gamma_{2\delta}}f(\xx)\left(1-\int_{\ONL}\bRd \d\yy\right)\d\xx\right)\right|\leq C\int_{\Gamma_{2\delta}}\left|f(\xx)\right|\d\xx.
\end{align*}

To estimate the integral in $\Gamma_{2\delta}$, we need a special cover of $\Gamma_{2\delta}$. Due to the smoothness of $\Gamma$ (at least $C^2$), each $\xx\in \Gamma$ has a neighborhood $B(\xx,r_\xx)$ 
such that $\p \ONL\cap B(\xx,r_\xx)$ coincides with the graph of a $C^2$ function and $\ONL\cap B(\xx,r_\xx)$ is the epigraph of this function after necessary relabeling and reorienting. These properties
allow straightening out the boundary.

Since $\Gamma$ is compact and $\bigcup\limits_{\xx\in \Gamma} B(\xx,\frac{1}{3}r_\xx)$ is an open cover of $\Gamma$, there exists a finite sub-cover $\bigcup\limits_{i=1}^m B(\xx_i,\frac{1}{3}r_{\xx_i})$. Take $r=\min\limits_{1\leq i\leq m}r_{\xx_i}$ and $\delta<\frac{1}{6}r$, 
we can assert $\bigcup\limits_{i=1}^m B(\xx_i,\frac{2}{3}r_{\xx_i})$ is an open cover of $\Gamma_{2\delta}$. That is because for each $\yy\in \Gamma_{2\delta}$, there exists a $\yy_{0}\in \Gamma$ such that $d(\yy,\Gamma)=\left|\yy-\yy_0\right|$. Let $\yy_0\in B(\xx_i,\frac{1}{3}r_{\xx_i})$, then 
\begin{align*}
	\left|\yy-\xx_i\right|\leq \left|\yy-\yy_0\right|+\left|\yy_0-\xx_i\right|< 2\delta+\frac{1}{3}r_{\xx_i}<\dfrac{2}{3}r_{\xx_i}.
\end{align*}
Now we have 
\begin{align*}
	\int_{\Gamma_{2\delta}}\left|f(\xx)\right|\d\xx\leq\sum_{i=1}^{m}\int_{B(\xx_i,\frac{2}{3}r_{\xx_i})\cap \Gamma_{2\delta}}\left|f(\xx)\right|\d\xx.
\end{align*}

Furthermore, for each $\xx_i$, it is rational to assume $B(\xx_i,r_i)\cap \Gamma$ is flat and $B(\xx_i,r_i)\cap \ONL$ is upper hemisphere, otherwise we can straighten the boundary. 
When $f\in C^1(\overline{\ONL})\cap H^1(\Omega)$ specially,
\begin{align*}
	&\int_{B(\xx_i,\frac{2}{3}r_{\xx_i})\cap \Gamma_{2\delta}}\left|f(\xx)\right|\d\xx\\
	\leq&\int_{B(\xx_i,\frac{2}{3}r_{\xx_i})\cap \Gamma_{2\delta}}\int_0^{2\delta}\left|f(\xx',x_n)\right|\d x_n\d {\xx'}\\
	\leq&\int_{B(\xx_i,\frac{2}{3}r_{\xx_i})\cap \Gamma_{2\delta}}\int_0^{2\delta}\left|f(\xx',0)+\int_{0}^{x_n}\pd{f}{x_n}(\xx',t)\d t\right|\d x_n\d {\xx'}\\
	\leq&\int_{B(\xx_i,\frac{2}{3}r_{\xx_i})\cap \Gamma_{2\delta}}\int_0^{2\delta}\left|f(\xx',0)\right|\d x_n\d \xx'+\int_0^{x_n}\int_{B(\xx_i,\frac{2}{3}r_{\xx_i})\cap \Gamma_{2\delta}}\int_0^{2\delta}\left|\pd{f}{x_n}(\xx',t)\right|\d x_n\d {\xx'}\d t\\
	\leq& C_i\delta\int_{B(\xx_i,k_i\delta)\cap \Gamma_{2\delta}}\left|f(\xx)\right|\d S_{\xx}+C_i\delta \int_{B(\xx_i,k_i\delta)\cap \ONL}\left|\nabla f(\xx)\right|\d {\xx}\\
	\leq& C_i\delta\norm{f}_{L^2(\Gamma)}+C_i\delta\norm{\nabla f}_{L^2(\ONL)}\\
	\leq& C_i\delta \norm{f}_{H^1(\ONL)}
\end{align*}
Take $C=m\max\limits_{1\leq i\leq m} C_i$,  we can get 
\begin{align*}
	\int_{\Gamma_{2\delta}}\left|f(\xx)\right|\d\xx\leq \sum_{i=1}^m\int_{B(\xx_i,\frac{2}{3}r_{\xx_i})\cap \Gamma_{2\delta}}\left|f(\xx)\right|\d\xx\leq C\delta \norm{f}_{H^1(\ONL)}.
\end{align*}
In addition, because $C^1(\overline{\ONL})$ is dense in $H^1(\ONL)$, the inequality above holds for $f\in H^1(\Omega)$. Now we can conclude 
\begin{align*}
	\left|\bar{f}\right|\leq C\delta \norm{f}_{H^1(\ONL)}\leq C\delta\norm{f}_{H^1(\Omega)}.
\end{align*} 

\section{Proof of Proposition \ref{PROP:TRUNCATION ERROR INTERFACE}}
\label{appendix:proof2}
For $u\in H^1(\Omega)\cap H^3(\OL)\cap H^3(\ONL)$, we define 
\begin{align*}
	\hat{r}_{\Gamma}(\xx)=\int_{\ONL}\bRd(u(\yy)-u(\xx))\d\yy+2\delta^2\pd{u}{\nn}(\xx)\int_{\Gamma}\bbRd \d S_\yy
\end{align*}
Noticing that 
\begin{align*}
	r_{\Gamma}(\xx)=\dfrac{1}{\wdd(\xx)}\hat{r}_{\Gamma}(\xx),
\end{align*}
if we can prove $\norm{\hat{r}_{\Gamma}}_{L^2(\Gamma)}\leq C\delta^2\norm{u}_{H^3(\ONL)}$, we can get 
\begin{align*}
	\norm{{r}_{\Gamma}}_{L^2(\Gamma)}=\left(\int_{\Gamma} \dfrac{1}{\wdd^2(\xx)}\hat{r}_{\Gamma}^2(\xx)\d\xx\right)^{\frac{1}{2}}\leq C\norm{\hat{r}_{\Gamma}}_{L^2(\Gamma)}\leq C\delta^2\norm{u}_{H^3(\ONL)}.
\end{align*}

We next estimate $\norm{\hat{r}_{\Gamma}}_{L^2(\ONL)}$. When $u\in C^3(\overline{\ONL})$, we have 
\begin{align*}
	&\int_{\ONL}\bRd(u(\yy)-u(\xx))\d\yy\\
	=&\int_{\ONL}\bRd(u(\yy)-u(\xx)-\nabla u(\xx)\cdot(\yy-\xx))\d\yy+\int_{\ONL}\bRd\left(\nabla u(\xx)\cdot(\yy-\xx)\right)\d\yy\\
	=&\int_{\ONL}\bRd(u(\xx)-u(\yy)-\nabla u(\xx)\cdot(\yy-\xx))\d\yy-2\delta^2\nabla u(\xx)\int_{\ONL}\nabla\bbRd\d\yy\\
	=&\int_{\ONL}\bRd(u(\xx)-u(\yy)-\nabla u(\xx)\cdot(\yy-\xx))\d\yy-2\delta^2\int_\Gamma \bbRd\nabla u(\xx)\cdot\nn(\yy)\d S_\yy.
\end{align*}
Thus,
\begin{align*}
	\hat{r}(\xx)&=\int_{\ONL}\bRd(u(\yy)-u(\xx)-\nabla u(\xx)\cdot(\yy-\xx))\d\yy+2\delta^2\int_{\Gamma}\bbRd \nabla u(\xx)\cdot(\nn(\xx)-\nn(\yy))\d S_\yy.
\end{align*}
The second term can be estimated as 
\begin{align*}
	&\int_\Gamma\left(2\delta^2\int_{\Gamma}\bbRd \nabla u(\xx)\cdot(\nn(\xx)-\nn(\yy))\d S_\yy\right)^2\d S_\xx\\
	\leq&C\delta^6\int_{\Gamma}\left(\int_{\Gamma}\bbRd \left|\nabla u(\xx)\right|\d S_\yy\right)^2\d S_\xx\\
	\leq&C\delta^6\int_{\Gamma}\left|\nabla u(\xx)\right|^2\left(\int_{\Gamma}\bbRd\d S_\yy\right)^2\d S_\xx\\
	\leq&C\delta^4\norm{\nabla u}_{L^2(\Gamma)}^2
	\leq C\delta^4\norm{u}_{H^3(\ONL)}^2.
\end{align*}
For the first term,
\begin{align*}
	&u(\yy)-u(\xx)-\sum_{i=1}^n\nabla^i u(\xx)(y_i-x_i)\\
	=&\int_0^1\dfrac{\d}{\d s_1}(u(\xx+s_1(\yy-\xx)))\d s_1-\sum_{i=1}^n\nabla^i u(\xx)(y_i-x_i)\\
	=&\sum_{i=1}^n\int_0^1\left(\nabla^i u(\xx+s_1(\yy-\xx))-\nabla^i u(\xx)\right)(y_i-x_i)\d s_1\\
	=&\sum_{i=1}^n\int_0^1\int_0^1\dfrac{\d}{\d s_2}\nabla^i u(\xx+s_2s_1(\yy-\xx))\d s_2\d s_1(y_j-x_j)\\
	=&\sum_{i,j=1}^n\int_0^1\int_0^1\nabla^i\nabla^j u(\xx+s_2s_1(\yy-\xx))s_1\d s_2\d s_1(y_i-x_i)(y_j-x_j),
\end{align*} 
Noticing that, with $\mathbf{z} = \xx + s(\yy-\xx)$,
\begin{align*}
	&\delta^4\int_{\Gamma}\int_{\ONL}\bRd\left|\nabla^i\nabla^j u(\xx+s(\yy-\xx))\right|^2\d\yy\d S_\xx\\
	\leq &\delta^4\int_{\Gamma}\int_{\ONL} \delta^{-n}\bar{R}\left(\frac{|\xx-\mathbf{z}|^2}{4s^2\delta^2}\right)\left|\nabla^i\nabla^j u(\mathbf{z})\right|^2\frac{1}{s^n}\d \mathbf{z}\d S_\xx\\
	=&\delta^4\int_{\Gamma}\int_{\ONL} \bar{R}_{s\delta}(\xx,\mathbf{z})\left|\nabla^i\nabla^j u(\mathbf{z})\right|^2\d \mathbf{z}\d S_\xx\\
	=&\delta^4\int_{\Gamma}\int_{\Gamma_{2s\delta}}\bar{R}_{s\delta}(\xx,\mathbf{z})\left|\nabla^i\nabla^j u(\mathbf{z})\right|^2\d \mathbf{z}\d S_\xx\\
	\leq&\delta^4\frac{1}{s\delta}\int_{\Gamma_{2s\delta}}\left|\nabla^i\nabla^j u(\mathbf{z})\right|^2\d \mathbf{z}\\
	\leq&\delta^4 \norm{u}_{H^3(\ONL)}.
\end{align*}
In the last inequality, we used an estimation
\begin{align*}
	\norm{u}_{L^2(\Gamma_{\sigma})}\leq C \sigma \norm{u}_{H^1(\ONL)}, \quad \text{for } u\in H^1(\ONL),\sigma<<1,
\end{align*}
which can be proved with a same method in Appendix \ref{appendix:proof1}.
Now we can get 
\begin{align*}
	\norm{\int_{\ONL}\bRd(u(\xx)-u(\yy)-\nabla u(\xx)\cdot(\yy-\xx))\d\yy}_{L^2(\ONL)}\leq C\delta^2\norm{u}_{H^3(\ONL)}.
\end{align*}
Combine these results, we conclude
\begin{align*}
	\norm{\hat{r}_\Gamma}_{L^2(\Gamma)}\leq C\delta^2\norm{u}_{H^3(\ONL)}.
\end{align*}
Because $C^3(\overline{\ONL})$ is dense in $H^3(\ONL)$, and we can verify the map $T: u\rightarrow \norm{\hat{r}_\Gamma}_{L^2(\Gamma)}$ is continuous in $H^3(\ONL)$,
this estimation also holds for $u\in H^3(\ONL)$.
\begin{remark}
	When we estimate the first term, we have in fact assume $\ONL$ is convex to help $\xx+s(\yy-\xx)$ make sense. In fact, this assumption is not indispensable.
	We can adopt the idea of parametrization introduced in \cite{shi2017convergence} to get the local convexity in the parameter space. Here we omit that complicated parametrization to make our 
	proof highlight the essence.
\end{remark}

\section{Proof of Lemma \ref{INTEGRAL ERROR}}
\label{appendix:proof3}
In section \ref{sec:convergence proof}, Lemma \ref{INTEGRAL ERROR} is stated following the notations in the context. Here we prove this lemma in a general bounded open region $U\subset \RR^n$ with a $C^2$ boundary $\partial U$.
\begin{figure}[h]
	\centering
	\begin{tikzpicture}[scale=2.3]
		\draw[->] (-2.6,0) -- (2.5,0) node[right] {$\RR^{n-1}$};
		\draw[->] (0,-0.4) -- (0,1.2) node[above] {$x^{(n)}$};
		
		\draw[domain=-2.6:2.5,smooth,variable=\x,black] plot ({\x},{0.2*\x*\x});
		
		\draw[domain=-2.6:-0.5,smooth,variable=\x,black] plot ({\x},{-0.8*\x-0.8});
		
		\filldraw[black] (-2, 0.8) circle (0.5pt) node[above right] {};
		\draw[black] (-2, 0.8) circle (0.6); 
		
		\begin{scope}
			\clip (-2,0.8) circle (0.6);
			\fill[red, opacity=2] 
			plot[domain=-2.5:2.5] (\x, {-0.8*\x-0.8}) --
			plot[domain=2.5:-2.5] (\x, {0.2*\x*\x}) -- cycle;
		\end{scope}
		\node at (1, 1) {$U$};
		\node at (2.3, 0.5) {$\varphi(x^{(1)},\cdots,x^{(n-1)})$};
		\node at (-2.2,0.6) {$(\xx',\varphi(\xx'))$};
		\node[below right] at (0,0) {$\xx_0$};
	\end{tikzpicture}
	\caption{A zoomed-in view of a neighborhood of $\xx_0$. In this region, the boundary of $U$ is the graph of a function $\varphi$ under a coordinate frame with origin $\xx_0$.
		The difference between the integral and $\frac{1}{2}$ in Lemma \ref{INTEGRAL ERROR} is exactly the volume of the red region.}
	\label{local coordinate frame}
\end{figure}
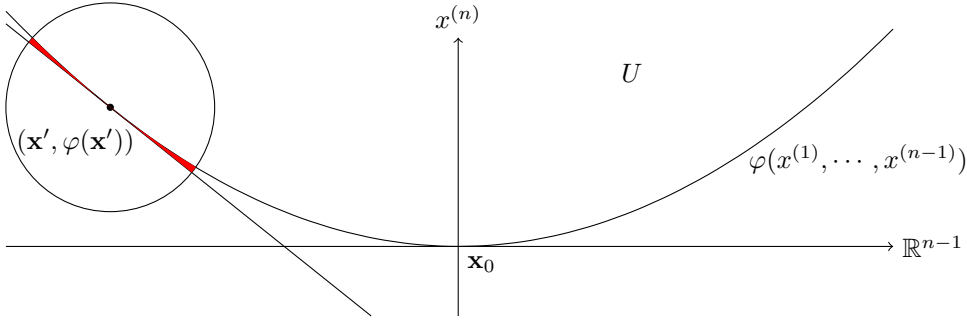

Since the boundary $\p U$ is $C^2$, at each point $\xx\in \p U$, there exists a positive constant $r_x$ and a $C^2$ function $\varphi:\RR^{n-1}\rightarrow \RR$ such that 
\begin{align*}
	B(\xx,r_\xx)\cap U=\left\{\zz\in B(\xx,r_\xx)\big|z_n>\varphi(z^{(1)},\cdots,z^{(n-1)})\right\}
\end{align*}
after relabeling and reorienting the coordinate frame if necessary. Now $\bigcup_{\xx\in\p U} B(\xx,r_\xx) $ is an open cover of $\p U$. 
Because $\p U$ is compact, we can find a finite sub-cover $\bigcup_{i=0}^m B(\xx_i,r_{\xx_i})$. In addition, Lebesgue's number lemma ensures a constant $\sigma>0$, such that for arbitrary $\xx\in \p U$, $B(\xx,\sigma)$ contains in some elements of this cover. 

Now we fix $\delta<\frac{1}{2}\sigma$. As shown in Figure \ref{local coordinate frame}, for $\xx\in \p U$, we assume $B(\xx,2\delta)\subset B(\xx_0,r_{\xx_0})$. In the local coordinate frame with origin $\xx_0$, the coordinate of $\xx$ is $(\xx',\varphi(\xx'))$, 
where the first $n-1$ coordinate components are abbreviated as $\xx'$. Additionally, the tangent hyperplane at $\xx$ splits $\RR^n$ into two parts $H^+$ and $H^-$. It doesn't matter to assume $(B(\xx,2\delta)\cap U) \subset H^+$.
Because $\bar{R}$ has compact support and the normalization (\ref{kernel normalization}), we have
\begin{align*}
	\left|\int_{U}\bRd\d\yy-\frac{1}{2}\right|
	=&\left|\int_{U}\bRd\d\yy-\int_{B(\xx,2\delta)\cap H^+}\bRd\d\yy\right|\\
	=&\left|\int_{B(\xx,2\delta)}\bRd \chi_{U}(\yy)\d\yy-\int_{B(\xx,2\delta)}\bRd\chi_{H^+}(\yy)\d\yy\right|\\
	\leq&\int_{B(\xx,2\delta)}\bRd\left|\chi_{U}(\yy)-\chi_{H^+}(\yy)\right|\d\yy\\
	\leq&C\delta^{-n}\int_{B(\xx,2\delta)}\left|\chi_{U}(\yy)-\chi_{H^+}(\yy)\right|\d\yy\\
	=&C\delta^{-n}\left|(B(\xx,2\delta)\cap H^+)-(B(\xx,2\delta)\cap U)\right|.
\end{align*}  
In Figure \ref{local coordinate frame}, $(B(\xx,2\delta)\cap H^+)-(B(\xx,2\delta)\cap U)$ is exactly the region filled with red. To estimate its measurement, 
we only need the height from $\p U$ to the tangent hyperplane. The normal direction at $\xx$ can be calculated as 
\begin{align*}
	\mathbf{v}=\left(\pd{\varphi}{x^{(1)}}(\xx'),\cdots,\pd{\varphi}{x^{(n-1)}}(\xx'),-1\right)^T.
\end{align*}
Hence, for $\yy\in B(\xx,2\delta)\cap \p U$, the projection of $\yy-\xx$ to $\mathbf{v}$ can be estimated as 
\begin{align*}
	\left|P_{\mathbf{v}}(\yy-\xx)\right|
	=&\left|\left\langle\yy-\xx,\frac{\mathbf{v}}{\left|\mathbf{v}\right|}\right\rangle\right|\\
	=&\frac{1}{|\mathbf{v}|}\left|\sum_{i=1}^{n-1}\pd{\varphi}{x^{(i)}}(\xx')(y^{(i)}-x^{(i)})-\left(\varphi(y^{(1)}\cdots,y^{(n-1)})-\varphi(x^{(1)}\cdots,x^{(n-1)})\right)\right|\\
	=&\frac{1}{\left|\mathbf{v}\right|}\left|\sum_{i,j=1}^{n-1}\dfrac{1}{2}\pd{^2\varphi}{x^{(i)}\p x^{(j)}}(\xi')(y^{(i)}-x^{(i)})(y^{(j)}-x^{(j)})\right|\\
	\leq& C(\xx_0,r_{\xx_0})\delta^2
\end{align*}
Here, the constant $C(\xx_0,r_{\xx_0})$ depends on the second order derivatives of $\varphi$ in $B(\xx_0,r_{\xx_0})$. Furthermore, because in $B(\xx,2\delta)$, the measurement on the tangent hyperplane is $O(\delta^{n-1})$,
we conclude 
\begin{align*}
	\left|\int_{U}\bRd\d\yy-\frac{1}{2}\right|\leq C(\xx_0,r_{\xx_0})\delta^{-n}\delta^2\delta^{n-1}=C(\xx_0,r_{\xx_0})\delta.
\end{align*} 
Take $C=\max\limits_{0\leq i\leq m} C(\xx_i,r_{\xx_i})$, we can get the required estimation
\begin{align*}
	\left|\int_{U}\bRd\d\yy-\frac{1}{2}\right|\leq C\delta.
\end{align*}

\bibliographystyle{abbrv}

\bibliography{ref}

\end{document}